\documentclass[11pt]{amsart}

\usepackage{amssymb}
\usepackage[utf8]{inputenc}
\usepackage[colorlinks=true, urlcolor=rltblue, citecolor=drkgreen, linkcolor=drkred] {hyperref}
\usepackage{xcolor}
\usepackage{pdfsync}
\definecolor{rltblue}{rgb}{0,0,0.4}
\definecolor{drkred}{rgb}{0.6,0,0}
\definecolor{drkgreen}{rgb}{0,0.4,0}
\usepackage{graphicx}
\usepackage[mathscr]{eucal} 

\usepackage{diagbox}

\usepackage{tabularray}
\usepackage{lmodern}
\usepackage[capitalize]{cleveref}
\usepackage{amssymb,amsmath,amsthm}
\usepackage{mathtools, MnSymbol}
\usepackage{setspace}
\usepackage{tikz-cd}
\usepackage[T1]{fontenc}
\usepackage[utf8]{inputenc}
\usepackage{textcomp} 
\usepackage{stmaryrd}
\usepackage{datetime}
\usepackage{todonotes}
\usepackage{mdframed}

\usepackage{thmtools}
\usepackage{pifont}

\newcommand{\domark}[1]{%
  \vbox to 0pt{
    \kern-\dp\strutbox
    \hbox{\smash{\llap{#1\kern1em}}}
    \vss
  }%
}

\providecommand{\tightlist}{%
	\setlength{\itemsep}{0pt}\setlength{\parskip}{0pt}}
\usepackage{thmtools}
\usepackage{enumitem}

\usepackage[backend=biber,url=false,maxbibnames=9]{biblatex}

\declaretheorem{theorem}
\declaretheorem[sibling=theorem]{lemma}
\declaretheorem[sibling=theorem]{proposition}
\declaretheorem[sibling=theorem]{corollary}
\declaretheorem[sibling=theorem,style=definition]{definition}

\declaretheorem[numberwithin=theorem,title=Claim]{claim}
\newcommand{\M}{\mathcal{M}}
\newcommand{\A}{\mathcal{A}}
\newcommand{\B}{\mathcal{B}}
\newcommand{\C}{\mathcal{C}}
\newcommand{\N}{\mathcal{N}}

\newcommand{\SR}{\text{SR}}

\newcommand{\mc}[1]{\mathcal{#1}}

\newcommand{\LR}{\Leftrightarrow}

\newcommand{\Pinf}[1]{\Pi^{\mathrm{in}}_{#1}}
\newcommand{\Sinf}[1]{\Sigma^{\mathrm{in}}_{#1}}
\newcommand{\Dinf}[1]{\Delta^{\mathrm{in}}_{#1}}
\newcommand{\dSinf}[1]{d\text{-}\Sigma^{\mathrm{in}}_{#1}}
\newcommand{\PA}{\mathrm{PA}}
\newcommand{\SSy}{\mathrm{SSy}}
\newcommand{\tp}{\mathrm{tp}}
\newcommand{\Th}{\mathrm{Th}}
\renewcommand{\phi}{\varphi}
\newcommand{\bigwwedge}{%
	\mathop{
		\mathchoice{\bigwedge\mkern-15mu\bigwedge}
		{\bigwedge\mkern-12.5mu\bigwedge}
		{\bigwedge\mkern-12.5mu\bigwedge}
		{\bigwedge\mkern-11mu\bigwedge}
	}
}

\newcommand{\bigvvee}{%
	\mathop{
		\mathchoice{\bigvee\mkern-15mu\bigvee}
		{\bigvee\mkern-12.5mu\bigvee}
		{\bigvee\mkern-12.5mu\bigvee}
		{\bigvee\mkern-11mu\bigvee}
	}
}
\newmdtheoremenv[backgroundcolor=cyan]{theorem-prove}{Theorem}[theorem]
\newmdtheoremenv[backgroundcolor=cyan]{lemma-prove}{Lemma}[theorem]
\newmdtheoremenv[backgroundcolor=cyan]{proposition-prove}{Proposition}[theorem]

\newmdtheoremenv[backgroundcolor=yellow!40]{theorem-check}{Theorem}[theorem]
\newmdtheoremenv[backgroundcolor=yellow!40]{lemma-check}{Lemma}[theorem]
\newmdtheoremenv[backgroundcolor=yellow!40]{proposition-check}{Proposition}[theorem]

\newcommand{\SSprank}{\mathrm{RankSpec}}

\def\hbar{{\bar{h}}}

\def\om{\omega}

\def\Si{\Sigma}

\def\A{\mathcal A}
\def\B{\mathcal{B}}
\def\C{\mathcal{C}}

\def\M{{\mathcal M}}

\def\N{{\mathcal N}}

\def\dSi{d\text{-}\Si}

\def\Lomom{L_{\om_1,\om}}

\newtheorem{thm}{Theorem}

\theoremstyle{remark}

\newtheorem{remark}[thm]{Remark}

\def\and{\mathrel{\&}}

\newcommand{\define}[1]{\textbf{#1}}
\newcommand{\ba}{\overline{a}}
\newcommand{\bb}{\overline{b}}
\newcommand{\bc}{\overline{c}}
\newcommand{\bd}{\overline{d}}
\title{Classifying the Complexities of Models of Arithmetic}
\author{David Gonzalez}
\author{Mateusz \L{}e\l{}yk}
\author{Dino Rossegger}
\author{Patryk Szlufik}

\address[Rossegger]{Institut f\"ur Diskrete Mathematik und Geometrie\\
  Technische Universit\"at Wien\\
  Wiedner Hauptstra{\ss}e 8-10\\
  1040 Wien\\
  AUSTRIA}
\email{\href{mailto:dino.rossegger@tuwien.ac.at}{dino.rossegger@tuwien.ac.at}}
\urladdr{\url{https://drossegger.github.io/}}
\address[\L{}e\l{}yk]{Faculty of Philosophy\\
  University of Warsaw\\
  Krakowskie Przedmieście 26/28\\
  00-927 Warsaw\\
  POLAND}
\email{\href{mailto: mlelyk@uw.edu.pl}{mlelyk@uw.edu.pl}}
\urladdr{\url{https://sites.google.com/uw.edu.pl/lelyk}}
\address[Gonzalez]{University of Notre Dame\\
Department of Mathematics\\
Hurley Hall, 255 Hurley, Notre Dame, IN 46556\\
  USA}
\email{\href{dgonza42@nd.edu}{dgonza42@nd.edu}}
\urladdr{\url{https://www.davidgonzalezlogic.com}}
\address[Szlufik]{Faculty of Mathematics, Informatics and Mechanics\\ University of Warsaw\\
  Krakowskie Przedmieście 26/28\\
  00-927 Warsaw\\
  POLAND}
\email{\href{mailto:p.szlufik@uw.edu.pl}{p.szlufik@uw.edu.pl}}
\urladdr{\url{https://sites.google.com/uw.edu.pl/szlufik/}}
\thanks{The work of the second author was supported by the National Science Centre (NCN) grant no. 2022/46/E/HS1/00452. The work of the third author was supported by the Austrian Science Fund (FWF) 10.55776/PIN1878224}

\begin{document}
\maketitle
\begin{abstract}
    We classify the possible Scott complexities for models of Peano arithmetic.
We construct models of particular complexities by first giving a complete Scott analysis of colored linear orderings and constructing models of Peano arithmetic from these colored orderings.
We also provide tight connections of certain Scott complexities with notions from the classical theory of models of Peano arithmetic, such as prime, finitely generated, and recursively saturated. 
This effort provides a powerful set of tools to understand the models of Peano arithmetic.
\end{abstract}

\section{Introduction}

It is straightforward to see that every finite structure $\A$ in finite vocabulary has a sentence $\phi_\A$ of finitary first-order logic so that the models of this sentence are precisely the structures isomorphic to it. Unfortunately, if we turn our attention to infinite structures, this fails. By a simple application of compactness, one can show that for most countable structures $\A$, even their first-order theory $\Th(\A)$ has countable models that are not isomorphic to $\A$; such theories fail to be what model theorists call \emph{countably categorical}. In some sense, a theory that is not countably categorical fails to characterize the structural properties of its models. Even among simple structures that are encountered in everyday mathematics, we can find examples of structures whose theories are not countably categorical. The example most central to modern-day mathematics is probably the standard model of arithmetic $(\mathbb N,+,\cdot)$. 

One could, of course instead take a more proof-theoretic approach by starting with a theory and asking whether this theory has a single model. However, in this case, the situation becomes even worse. For example, by Gödel's famous incompleteness theorem, \emph{Peano arithmetic} ($\PA$ for short)---the standard axiomatization of number theory---is incomplete and possesses no recursive completions. One can then show that none of its completions is countably categorical and, indeed, every such completion has $2^{\aleph_0}$ models.

But also when considered from other angles, Peano arithmetic is a consummately wild theory. For example, model theory is a subfield of mathematical logic that analyzes mathematical structures based on their first-order theories. Modern model theory, or neo-stability theory, concerns itself a great deal with characterizing and capturing various notions of mathematical tameness. However, due to G\"odel's initial observations about the coding power of Peano arithmetic, this theory and its models stand outside any notion of model-theoretic tameness. See~\cite{forkanddivide} for a graphical representation of tameness properties used in modern model theory; Peano arithmetic is a member of the delta quadrant, so it escapes all studied notions of tractability.

Despite the inherent difficulty in studying Peano arithmetic and the failure of most model-theoretic techniques to obtain insight into structural properties of its models, researchers have still made impressive efforts and proved exceptional results. In fact, the study of models of Peano arithmetic has been one of the most vital efforts in logic for the nearly one hundred years since Peano introduced the axiomatization. In the study of models of arithmetics, researchers usually combine model-theoretic techniques with the special properties of models of Peano arithmetic, such as their immense coding power, to gain new structural insights, see~ \cite{KS} for an overview of results in this area.

In this paper, we take a different approach to study the structural properties of models of Peano arithmetic. Instead of finitary first-order logic, we turn to the infinitary logic $\Lomom$, allowing countable conjunctions and disjunctions in our language. Scott~\cite{Sco65} showed that, in contrast to the situation with first-order logic, we can classify countable structures up to isomorphism in the logic $\Lomom$. In this logic, every countable structure $\A$, including $(\mathbb{N},+,\cdot)$ and other models of Peano arithmetic, has a $\Lomom$-sentence $\phi_\A$ the countable models of which are precisely the structures isomorphic to $\A$. We call such a sentence a \define{Scott sentence}.
Scott's work lies in the tradition of computable structure theory, not first-order model theory, and his result was the first in the field we now call \define{Scott analysis}.
In this article, we use recent developments in Scott analysis to reframe the theory of models of Peano arithmetic that had previously only been considered in a more ad hoc manner and give a complete analysis of the structural complexities of its models.

In Scott analysis, the structural complexity of a model is measured by the complexity of its simplest Scott sentence. We can naturally extend the classical quantifier-alternation rank to $\Lomom$ and through this obtain various complexity measures on sentences, and therefore their models. The best-known of these measures are the various notions of \emph{Scott rank} emerging from Scott's proof. To define this notion let us first recall the definitions of quantifier-rank in $\Lomom$.
\begin{definition}
	We let $\Pinf{0}=\Sinf{0}$ be the class of finitary quantifier-free formulas. For countable ordinals $\alpha>0$ and $\phi\in \Lomom$ we let 
	\begin{itemize}
		\item $\phi\in \Sinf{\alpha}$ if it is logically equivalent to a formula of the form $\bigvvee_i \exists \bar x\psi_i(\bar x)$ where $\psi_i\in \Pinf{\gamma_i}$ for some $\gamma_i<\alpha$;
		\item $\phi\in \Pinf{\alpha}$ if it is logically equivalent to a formula of the form $\bigwwedge_i \forall\bar x\psi_i(\bar x)$ where $\psi_i\in \Sinf{\gamma_i}$ for some $\gamma_i<\alpha$.
	\end{itemize}
	Furthermore, $\phi\in \dSinf{\alpha}$ if there is $\psi\in \Sinf{\alpha}$ and $\theta\in \Pinf{\alpha}$ so that $\phi$ is logically equivalent to $\psi\land \theta$.
\end{definition}
It is not hard to see that neither $\Pinf{\alpha}$ nor $\Sinf{\alpha}$ are subsets of each other and that $\dSinf{\alpha}\subseteq \Pinf{\alpha+1}\cap \Sinf{\alpha+1}$.
We can now introduce the notions of Scott rank that we will use. Our definitions follow Montalbán~\cite{MonSR}, who showed that these notions of rank interact nicely with various other complexity measures~\cite{MonSR}. Montalbán defined the \define{Scott rank} of $\A$ to be the least ordinal $\alpha$ so that $\A$ has a $\Pinf{\alpha+1}$ Scott sentence and showed that among others this rank corresponds with the following "internal complexity measure": It is the least $\alpha$ so that every $\Pinf{\alpha}$-type realized in $\A$ is supported by a $\Sinf{\alpha}$-formula, or equivalently, so that every automorphism orbit in $\A$ is $\Sinf{\alpha}$-definable. Connected to this is the \define{parametrized Scott rank} of $\A$, the least $\alpha$ so that $\A$ has a $\Sinf{\alpha+2}$ Scott sentence. Montalbán showed that a structure has parametrized Scott rank $\alpha$ if and only if there is a parameter $\ba$ from $\A$ so that $(\A,\ba)$ has Scott rank $\alpha$. Alvir, Greenberg, Harrison-Trainor, and Turetsky~\cite{AGNHTT} showed that every countable structure $\A$ has an optimal Scott sentence and that this sentence is of complexity either $\Sinf{\alpha}$, $\Pinf{\alpha}$, or $\dSinf{\alpha}$ for some least ordinal $\alpha$. They called this ordinal the \define{Scott (sentence) complexity} of $\A$. Strictly speaking, Scott complexity is not a rank, as its range is not linearly ordered since there are structures with $\Pinf{\alpha}$ but no $\Sinf{\alpha}$ Scott sentences and vice versa. However, it has the advantage of being finer than the notions of rank defined above and is related to them via the complexity of the automorphism orbit of the parameters involved in the parametrized Scott rank. The relationship between these ranks can be found in~\cite{MBook}, while the case for limit ordinals was recently resolved in~\cite{GR23} using the additional concept of "unstable sequences".

To obtain a structural complexity measure for theories from these considerations we consider the collection of Scott ranks realized by their models. This collection is called the theory's \define{Scott spectrum} and defined as
\[ \SSprank(T)=\{\alpha : (\exists \A\models T) \text{ Scott rank of $\A$ is $\alpha$}\}
\]
We could define Scott spectra also for Scott complexity, but we opt to go one step further. Instead of just asking whether there are models of a theory of a certain complexity we also want to know how many and thus define the following.
\begin{definition}
	Given a theory $T\in\Lomom$ and a complexity $\Gamma\in\{\Sinf{\alpha},\dSinf{\alpha},\Pinf{\alpha}\}_{\alpha\in\omega_1}$, the \define{Scott complexity counting function} $I$ is defined by
    \[ I(T,\Gamma) = \# \{\mc{M}\models T\mid SSC(\mc{M})=\Gamma\} .\]
\end{definition}

The goal of this article is to give a comprehensive Scott analysis of Peano arithmetic $\PA$ and its completions by calculating their counting functions.
This involves understanding interactions between the classical theory of models of Peano arithmetic and Scott analytic notions.
Montalbán and Rossegger~\cite{MR23} calculated $\SSprank(\PA)$ and $\SSprank(T)$ for all completions $T$ of $\PA$. They showed that the standard model $\mathbb N$ is the unique model of $\PA$ that has finite Scott rank (indeed it has Scott rank $1$) and proved the following characterization.
\begin{theorem}[\cite{MR23}]\label{thm:monross}
    Let $T$ be a completion of $\PA$.
		\begin{enumerate}
			\item If $T=\Th(\mathbb N)$, then $\SSprank(T)=\{1\}\cup \{\alpha<\omega_1: \alpha>\omega\}$.
			\item If $T\neq\Th(\mathbb N)$, then $\SSprank(T)=\{\alpha<\omega_1: \alpha\geq\omega\}$.
		\end{enumerate}
		Thus, $\SSprank(\PA)=\{1\}\cup\{ \alpha<\omega_1: \alpha\geq \omega\}$.
\end{theorem}
At the end of~\cite{MR23}, they asked several questions and in this article we answer them all. 
\begin{enumerate}
	\item We show that no non-atomic model of Peano arithmetic has Scott rank $\omega$ (\cref{thm_srundefinMOPA} answering~\cite[Question 1]{MR23}),
	\item that no model of $\PA$ has Scott complexity $\Pinf{\omega}$ (\cref{cor_noPiomega} answering~\cite[Question 3]{MR23}),
	\item and give a comprehensive analysis of the counting function for completions of $\PA$ (\cref{table:results} answering~\cite[Question 2]{MR23}).
\end{enumerate}

\begin{table}[ht]
	\begin{tblr}{colspec={|Q[c,m]|Q[c,m,$]|Q[c,m,$]|Q[c,m,$]|Q[l,m]|}}
		\hline
		\diagbox[innerwidth=4cm]{SSC}{Th}  & \PA & \Th(\mathbb N) & T\neq \Th(\mathbb N) & Reference\\
      \hline\hline
		$\Pinf{2}$ & 1 & 1 & 0 & $\mathbb N$\\
		other finite complexities& 0 & 0 &0 &\cite{MR23}\\
		$\Pinf{\omega}$ & 0 & 0 & 0 &\cref{cor_noPiomega}\\
		$\Pinf{\omega+1}$ & 2^{\aleph_0} & 0 & 1 &{atomic models \\\cref{cor_noPiomega}}\\
		$\Sinf{\omega+1}$ & 0 &0 &0  & \cref{cor_noSigmaomega+1}\\
		$\dSinf{\omega+1}$ & 2^{\aleph_0}& 2^{\aleph_0} & 2^{\aleph_0} & {fin.\ generated non-atomic\\\cref{prop_fingenchar}}\\
		$\Sinf{\omega+2}$ & 0 & 0 & 0 & \cref{cor_noSigmaomega+2}\\
		$\Pinf{\omega+\alpha}$ $\alpha>1$ & 2^{\aleph_0} & 2^{\aleph_0} & 2^{\aleph_0} & \cref{thm:transferedcomplexities}\\
		$\dSinf{\omega+\alpha}$ successor $\alpha>1$ &  2^{\aleph_0} & 2^{\aleph_0} & 2^{\aleph_0} & \cref{thm:transferedcomplexities}\\
		$\Sinf{\omega+\alpha}$ successor $\alpha>2$ & 2^{\aleph_0} & 2^{\aleph_0} & 2^{\aleph_0} & \cref{thm:transferedcomplexities}\\
			\hline
    \end{tblr}
    \caption{$I(T,\Gamma)$ for $\PA$ and its complete extensions. Complexities $\Gamma$ not
    in the table are impossible for structures in general.}
    \label{table:results}
  \end{table}

One approach to facilitate a Scott analysis of a theory like Peano arithmetic is to use a notion of reduction that preserves Scott rank, reduce a well-understood class of structures to your theory, and then use the structures in the image as examples for the Scott sentence complexities. This is, in essence, what Montalbán and Rossegger~\cite{MR23} did to characterize the Scott ranks of $\PA$. They showed that finite Scott ranks are not possible and then reduced the class of linear orderings to the canonical $\omega$-jumps of models of $\PA$.
The issue with their approach is that if we sharpen our lens to distinguish Scott sentence complexities, then not every Scott complexity occurs in the class of linear orderings. 

To overcome this, we turn to structures slightly more complicated than linear orderings, \emph{colored linear orderings}, i.e., linear orderings together with a partition of their universe in finitely or countably many subsets. We provide a comprehensive Scott analysis of this class in \cref{sec:coloredlo}. In \cref{sec:coloredtomodels} we adapt the reduction from~\cite{MR23} to reduce colored linear orderings to models of a fixed completion of $\PA$ and show that our reduction preserves Scott sentence complexities. At last, in \cref{sec:missing} we provide an analysis of the Scott sentence complexities not covered by our reduction and of models of $\PA$ that have seen interest in the literature such as finitely generated models and short recursively saturated models.

\section{Colored Linear Orderings}\label{sec:coloredlo}
In what follows, we will interpret colored linear orderings in models of Peano arithmetic.
This will allow us to push forward Scott analysis results about colored linear orderings to models of Peano arithmetic.
For this to be useful, however, we must first establish facts about colored linear orderings. 

For $k\leq \aleph_0$, a \define{$k$-colored linear ordering} $\mc{L}$ is a structure in the vocabulary $\{\leq,\{P_i\}_{i\in k}\}$ such that $\leq$ is a linear ordering and the $P_i$ partition the universe of $\mc{L}$. More formally, $\mc{L}$ satisfies the axioms $\forall x\;\bigvvee_{i\in k} P_i(x)$ and $\forall x \; \bigwwedge_{i,j\in k} \lnot \big( P_i(x) \land P_j(x) \big)$.
We say that $x\in \mc L$ has color $i$ if $P_i(x)$ and that $\mc L$ is \define{colored} if it is $k$-colored for some $k\leq \aleph_0$.

Note that if $k\leq m$, the identity function serves as an effective bi-interpretation from linear orderings with $k$ colors into linear orderings with $m$ colors. 
Furthermore, the function that takes the linear ordering $\mc{L}$ and adds every point into $P_0$ defines an effective bi-interpretation from linear orderings to linear orderings with 1 color.
Because this interpretation is surjective, we can treat linear orderings and linear orderings with 1 color as essentially the same theories and move back and forth along this bi-interpretation freely. 

The following definitions will be useful.
\begin{definition}\label{def:onecoloredlos}
	If $\mc L$ is a linear ordering and $i\in\omega$, then we write $\mc L_i$ for the monochromatic colored linear ordering where every point has color $i$.
\end{definition}
\begin{definition}\label{def:colorshuffle}
	Given $A\subseteq \omega$, the \define{color shuffle of $A$}, $Sh_c(A)$, is the structure with order type $\eta$ where each color in $A$ is dense (i.e., if $q_1\leq q_2$, then for all $i\in A$, there is $q$, $q_1<q<q_2$ so that $P_i(q)$).\end{definition}
The above definition says "the" color shuffle instead of "a" color shuffle because all color shuffles of $A$ are isomorphic.
This follows by a standard back-and-forth argument which we include below for completeness.
\begin{lemma}\label{lem:shufiso}
    For $\mc L$ and $\mc K$ color shuffles of $A$, if $\bar{a}\in\mc{L}$ and $\bar{b}\in \mc{K}$ are such that $qftp(\bar{a})=qftp(\bar{b})$, then
     $(\mc{L},\bar{a})\cong(\mc{K},\bar{b}).$
\end{lemma}

\begin{proof}
    Consider the following set of tuples:
    \[ B=\{\bar{c},\bar{d}\in\mc{L}^{<\omega}\times\mc{K}^{<\omega} \mid qftp_{\bar{a}}(\bar{c})=qftp_{\bar{b}}(\bar{d})\}.\]
    We claim that this is a back-and-forth set.
    Given $\bar{c},\bar{d}\in B$, and $p\in\mc{L}$ we produce a $q\in\mc{K}$ such that $qftp_{\bar{a}^\frown\bar{c}}(p)=qftp_{\bar{b}^\frown\bar{d}}(q)$.
    By symmetry, this is enough to demonstrate the claim.
    Say that $c_i<p<c_{i+1}$ where $\bar{a}^\frown\bar{c}=c_1<\cdots<c_m$, $c_0=-\infty$ and $c_{m+1}=\infty$.
    If we let $\bar{b}^\frown\bar{d}=d_1<\cdots<d_m$, $d_0=-\infty$ and $d_{m+1}=\infty$, it is enough to show that there is $q$ such that $d_i<q<d_{i+1}$ and with the same color as $p$.
    However, this follows immediately from the fact that every color is dense in $\mc{K}$.
    This means that $B$ is a back and forth set, so $(\mc{L},\bar{a})\cong(\mc{K},\bar{b}).$
\end{proof}

It follows that these colored shuffles are simple in terms of the Scott analysis.

\begin{corollary}\label{cor:ShufflePi2}
	The Scott complexity of $Sh_c(A)$ is $\Pinf{2}$.
\end{corollary}

\begin{proof}
    The above proposition shows that the automorphism orbit of a tuple in $Sh_c(A)$ is given by the order of the tuple and the colors of the elements in the tuple.
    This is quantifier-free definable, which means that $\SR(Sh_c(A))=1$.
    In other words, it has a $\Pinf{2}$ Scott sentence.
    As $Sh_c(A)$ is an infinite relational structure, it cannot have a $\Sinf{2}$ Scott sentence.
    (This follows from \cite{AGNHTT} Theorem 5.1.)
    These two facts together yield the desired result.

\end{proof}

\begin{corollary}\label{cor:self-sim}
    Any non-trivial interval of $Sh_c(A)$ without endpoints is isomorphic to $Sh_c(A)$.
\end{corollary}

\begin{proof}
 If $X$ is a non-trivial interval of $Sh_c(A)$, then all colors in $A$ are dense in $X$.
 If $X$ has no endpoints, $Sh_c(A)\cong X$ follows from \cref{lem:shufiso}.
\end{proof}

Our first fact points out that, unlike linear orderings, there is a colored linear ordering of every realizable Scott complexity.
Due to the results of \cite{GR23} and \cite{GHTH}, this amounts to showing that there is a colored linear ordering of complexity $\Sinf{3}$.

We will use the (standard asymmetric) back-and-forth relations
to demonstrate this claim.
\begin{definition}
The \textit{standard asymmetric back-and-forth relations} $\leq_\alpha$, for a countable ordinal $\alpha < \omega_1$, are defined by:
    \begin{itemize}
        \item $(\mc{M},\bar{a}) \leq_0 (\mc{N},\bar{b})$ if $\bar{a}$ and $\bar{b}$ satisfy the same quantifier-free formulas from among the first $|\bar{a}|$-many formulas.
        \item For $\alpha > 0$, $(\mc{M},\bar{a}) \leq_\alpha (\mc{N},\bar{b})$ if for each $\beta < \alpha$ and $\bar{d} \in \mc{N}$ there is $\bar{c} \in \mc{M}$ such that $(\mc{N},\bar{b} \bar{d}) \leq_\beta (\mc{M},\bar{a} \bar{c})$.
    \end{itemize}
We define $\bar{a} \equiv_\alpha \bar{b}$ if $\bar{a} \leq_\alpha \bar{b}$ and $\bar{b} \leq_\alpha \bar{a}$.
\end{definition}

These relations are usually conceptualized as a game between the $\forall$-player who moves first and picks a tuple of elements in alternating structures and the $\exists$-player who moves in response to this choice of tuple in the other structure.
The $\exists$-player attempts to copy the behavior exhibited by the $\forall$-player's choices and wins if they can do this until an ordinal clock expires.
The relation $\leq_\alpha$ holds if the $\exists$-player has a winning strategy in this game.
These following theorem connects these relations with infinitary logic.

\begin{theorem}[\cite{Karp}]
For any non-zero ordinal $\alpha$, structures $\mc{M}$ and $\mc{N}$ and tuples $\bar{a}\in\mc{M}$ and $\bar{b}\in\mc{N}$, the following are equivalent:
    \begin{enumerate}
	\item $(\mc{M},\bar{a})\leq_\alpha (\mc{N},\bar{b})$.
	\item Every $\Pinf{\alpha}$ formula true about $\bar{a}$ in $\mc{M}$ is true about $\bar{b}$ in $\mc{N}$.
	\item Every $\Sinf{\alpha}$ formula true about $\bar{b}$ in $\mc{N}$ is true about $\bar{a}$ in $\mc{M}$.
    \end{enumerate} 
\end{theorem}

We denote intervals within a colored linear ordering as follows.
\begin{definition}
    If $\mc{L}\models LO_k$ with $2\leq k\leq\aleph_0$ colors and $a,b\in\mc{L}$ let
    \[(a,b)_\mc{L}=\{x\in\mc{L}|a<x<b\}.\]
    We also let $\mc{L}_{>x}=(x,\infty)_\mc{L}$ and $\mc{L}_{<x}=(-\infty,x)_\mc{L}$
    If it is clear from the context what ordering $a$ and $b$ are elements of, the subscript $\mc{L}$ may be dropped.

\end{definition}

The back-and-forth relations for colored linear orderings obey the following rules that make them easier to calculate.

\begin{lemma}\label{lem:colorHelper}
    Let $LO_k$ be the theory of linear ordering with $2\leq k\leq\aleph_0$ colors. Let $\mc{L},\mc{K}\models LO_k$. Let $\bar{a}=a_1<\cdots<a_n\in\mc{L}$ and $\bar{b}=b_1<\cdots<b_n\in\mc{K}$ and by convention let $a_0=b_0=-\infty$ and $a_{n+1}=b_{n+1}=\infty$.
    \begin{enumerate}
        \item $(\mc{L},\bar{a})\leq_\alpha (\mc{K},\bar{b})$ if and only if for each $i\leq n$, $a_i$ has the same color as $b_i$ and $(a_i,a_{i+1})_{\mc{L}}\leq_\alpha (b_i,b_{i+1})_{\mc{K}}$.
        \item $\mc{L}\leq_1\mc{K}$ if and only if every finite, increasing sequence of colors in $\mc{K}$ is also present in $\mc{L}$.
    \end{enumerate}
    
\end{lemma}

\begin{proof}
We begin by showing the forward direction of Item (1).
Say that for some $i$, $(a_i,a_{i+1})_{\mc{L}}\not\leq_\alpha (b_i,b_{i+1})_{\mc{K}}$ witnessed by some formula $\psi$.
Relativizing the quantifiers in $\psi$ to being between the $i^{th}$ and $(i+1)^{st}$ elements of the tuple of parameters directly gives a $\hat{\psi}$ that witnesses $(\mc{L},\bar{a})\not\leq_{\alpha} (\mc{K},\bar{b})$.
Also, if $a_i$ and $b_i$ are different colors, $(\mc{L},\bar{a})\not\leq_0 (\mc{K},\bar{b})$ follows immediately.

The backward direction of Item (1) is a transfinite induction with the only interesting case being the successor case.
Say that $a_i$ has the same color as $b_i$ and $(a_i,a_{i+1})_{\mc{L}}\leq_{\beta+1} (b_i,b_{i+1})_{\mc{K}}$ for all $i\leq n$.
Consider $\bar d$ with $\bar d_i\subseteq \bar d$ in $(b_i,b_{i+1})$. Then there is $\bar c_i\in (a_i,a_{i+1})$ such that $(b_i,b_{i+1},\bar d_i)\leq_\beta ( a_i, a_{i+1},\bar c_i)$. 
For $i\leq n$ let $\bar{c}_i$ be the elements of $\bar{c}$ in $(b_i,b_{i+1})$.
Writing $\bar{a},\bar{c}=\bigcup \bar c_i$ as an increasing sequence $\bar{p}$ and $\bar{b},\bar{c}=\bigcup \bar d_i$ as an increasing sequence $\bar{q}$, the induction hypothesis gives that
\[(\mc{L},\bar{a},\bar{c})=(\mc{L},\bar{p})\geq_\beta(\mc{K},\bar{q})=(\mc{K},\bar{b},\bar{d}),\]
as desired.

Item (2) follows directly by observing that the quantifier-free type of a tuple is entirely determined by its ordering and the colors of the elements.
\end{proof}

\begin{theorem}\label{thm:allComplexities}
	There is a $2$-colored linear ordering that has Scott complexity $\Sinf{3}$.
    In particular, every Scott complexity is realizable by $k$-colored linear orderings for $k\geq 2$.
\end{theorem}

\begin{proof}
    Consider the linear ordering $\mc{L}=\eta_0+1_0+Sh_c(0,1)$.
    We claim that $\mc{L}$ is the desired order.
    
    We first show that $\mc{L}$ has a $\Sinf{3}$ Scott sentence.
    For this, it is enough to show that $\SR_p(\mc{L})=1$.
    Take the parameter $p$ to be the "$1_0$" listed in the sum.
    Automorphisms of $(\mc{L},p)$ fix the elements below and above $p$ setwise, so it is enough to define the automorphism orbits of tuples above and below $p$.
		However, we know that this can be done using $\Sinf{1}$-formulas by \cref{cor:ShufflePi2} and the fact that $\eta_0=Sh(0)$.
    Therefore, $\mc{L}$ has a $\Sinf{3}$ Scott sentence.

    We now show that this is the least complex Scott sentence for $\mc L$, or, in other terms, that $\SR(\mc{L})=3$.
		For this, it is enough to show that the element $p$ defined above does not have a $\Sinf{2}$-definable automorphism orbit. 
    To show this, we demonstrate that any $q\in \eta_0$ has the property that $q\leq_2 p$.
    As $p$ and $q$ are the same color, we need only show that $\mc{L}_{<q}\leq_2\mc{L}_{<p}$ and that $\mc{L}_{>q}\leq_2\mc{L}_{>p}$.
    Note that $\mc{L}_{<q}\cong\eta_0\cong\mc{L}_{<p}$, so it is enough to show that
    \[\mc{L}_{>q}\cong \mc{L} \leq_2 Sh_c(0,1) \cong \mc{L}_{>p},\]
    where the written isomorphisms follow from \cref{cor:self-sim}.
    Let $\iota:Sh_c(0,1)\to \mc{L}$ be the final embedding implicit in writing $\mc{L}=\eta_0+1_0+Sh_c(0,1)$.
    Given a play of $\bar{a}\in Sh_c(0,1)$ in the back-and-forth game, the $\exists$-player can respond with $\iota(\bar{a})$.
    Write $a_0=-\infty<a_1<\cdots<a_n<a_{n+1}=\infty$.
    We need to show for all $i<n$ that $(\iota(a_i),\iota(a_{i+1}))\geq_1 (a_i,a_{i+1}).$
    It follows from \cref{cor:self-sim} that for $i>0$ $(\iota(a_i),\iota(a_{i+1}))\cong Sh_c(0,1)\cong (a_i,a_{i+1}).$
    Therefore, we need only show that 
    \[\mc{L}_{<\iota(a_1)}\cong \mc{L} \geq_1 Sh_c(0,1) \cong Sh_c(0,1)_{<a_1},\]
    where the written isomorphisms again follow from \cref{cor:self-sim}.
    Observe that every increasing sequence of the colors $0$ and $1$ can be found in $Sh_c(0,1)$ because each color is dense.
    By \cref{lem:colorHelper}, this demonstrates that  $\mc{L} \geq_1 Sh_c(0,1)$, as desired.
    Putting it all together, we get that $q\leq_2 p$, so $\SR(\mc{L})=3$. Combining this with the fact proven above that $\SR_p(\mc{L})=1$ gives that the Scott complexity of $\mc{L}$ is $\Sinf{3}$.

    The second sentence of the theorem now follows from the fact that there are uncolored linear orderings of every other possible Scott complexity by Table 2 of \cite{GR23} along with Theorem 1.3 of \cite{GHTH}.
    These orderings can be effectively bi-interpreted into linear orderings with $k$ colors via the map $\mc{L}\mapsto \mc{L}_0$.
\end{proof}

\cref{thm:allComplexities} means that, in terms of Scott analysis, being able to interpret colored linear orderings is more powerful than being able to interpret linear orderings without colors.
There are still some limits on the power of such an interpretation that come from limits on the simple colored linear orderings.
For example, (as we see below) there are only countably many colored linear orderings with Scott complexity $\Pinf{2}$ when there are finitely many colors, and they all have a very particular form.
We will now calculate the complexity counting function for the theory of linear orderings and finitely colored linear orderings, which will make these limits explicit.
We will use the following notion in our analysis (see \cite[Definition 1.8]{GR23}).

\begin{definition}[\cite{GR23}]
    Given a structure $\mc{A}$ and limit ordinal $\lambda$, a \define{$\lambda$-sequence} in $\mc{A}$ is a sequence of tuples $(\bar{y}_i\in\mc{A})_{i
			\in
		\omega}$ such that for some increasing fundamental sequence $(\alpha_i)_{i\in\omega}$ for $\lambda$ we have
    \[\bar{y}_i \equiv_{\alpha_i} \bar{y}_{i+1}. \]
    The $\lambda$-sequence is called \define{unstable} if, in addition
    \[\bar{y}_i \not\equiv_{\alpha_{i+1}} \bar{y}_{i+1}.\]

\end{definition}

The following theorem is not explicitly stated elsewhere, but isolates the previously known importance of this notion.
\begin{theorem}\label{thm:internalSC}
Fix a structure $\mc{M}$.
The Scott complexity of $\mc{M}$ is determined by
    \begin{enumerate}
        \item $\SR(\M)$
        \item $\SR_p(\M)$
        \item the existence (or non-existence) of unstable $\lambda$-sequences in $\mc{M}$
        \item the existence (or non-existence) of unstable $\lambda$-sequences over parameters in $\mc{M}$.
    \end{enumerate}
\end{theorem}

\begin{proof}
    It follows from the analysis in \cite{AGNHTT} (see for example \cite[Table 1]{MBook}) that if $\SR_p(\M)$ is a successor ordinal then the Scott complexity of $\mc{M}$ is determined by just $\SR(\M)$ and $\SR_p(\M)$.
    The case where $\SR_p(\M)$ is a limit ordinal $\lambda$ is examined in~\cite[Lemma 1.9]{GR23}.
    If $\SR(\M)=\SR_p(\M)=\lambda$ the two possible Scott complexities of $\Pi_\lambda$ and $\Pi_{\lambda+1}$ are distinguished by the existence of unstable $\lambda$-sequences in $\mc{M}$.
    If $\SR_p(\M)=\lambda$ and $\SR(\M)=\lambda+1$ the two possible Scott complexities of $\Si_{\lambda+1}$ and $\dSi_{\lambda+1}$ are distinguished by the existence of unstable $\lambda$-sequences over parameters in $\mc{M}$.
    The final possibility that $\SR_p(\M)=\lambda$ and $\SR(\M)=\lambda+2$ gives that the Scott complexity of $\M$ is $\Si_{\lambda+2}$, which follows from the analysis in \cite{AGNHTT}, and does not need the notion of unstable $\lambda$-sequences to be understood.
\end{proof}

We will now show how to combine various results from recent literature to obtain the complexity counting function for the class of linear orderings.
\begin{theorem}\label{thm:nocolors}
    Let $LO$ be the theory of linear orderings. Then
        \begin{enumerate}
        \item $I(LO,\Pinf{i}) = 1$ if $i\leq 2$,
        \item $I(LO,\Sinf{i})=0$ if $i\leq 3$,
	\item $I(LO,\text{d-}\Sinf{i}) = \aleph_0$ if $i\leq 3$,
        \item $I(LO,\Pinf{3})=\aleph_0$,
        \item $I(LO,\Sinf{4})=\aleph_0$,
	\item $I(LO_k,\Gamma) = 2^{\aleph_0}$ otherwise.
    \end{enumerate}
\end{theorem}

\begin{proof}
	For complexities $\Gamma\leq\Sinf{4}$, there are only countably many orderings of complexity $\Gamma$ and these were explicitly characterized in~\cite[Theorem 3.7]{GR23}.
    Cases (1)-(5) in our theorem are all immediate corollaries of this classification.
		Hence, we focus on case (6) for $\Gamma\geq \Pinf{4}$
		It follows from \cite[Proposition 3.11]{GR23} that there are continuum many linear orderings $\{L_c\}_{c\in 2^\omega}$ of Scott complexity $\Pinf{4}$.
		Note the following two structural properties of each $L_c$ that follow from the proof of \cite[Proposition 3.11]{GR23}:
    \begin{itemize}
        \item $L_c$ has no last element,
        \item $L_c$ has no maximal successor chain of length $2$.
    \end{itemize}
		Between~\cite{GR23} and~\cite[Theorem 1.3]{GHTH} a linear ordering of complexity $\Gamma$ is constructed for each $\Gamma\geq \Pinf{4}$.
    It is straightforward (albeit tedious) to check that for each $\Gamma\geq \Pinf{4}$ there is a construction provided in these sources that produces a linear ordering $P_\Gamma$ with
    \begin{itemize}
        \item the Scott complexity of $P_\Gamma$ is $\Gamma$
        \item $P_\Gamma$ has no maximal successor chain of length 2,
        \item $P_\Gamma$ has no first element.
    \end{itemize}
    For each $c$ and $\Gamma$, consider $L_c+2+P_\Gamma$.
    Let the written "2" be given by parameters $\bar{a}$.
		Combining the facts provided above we get that $\bar a$ is the only maximal successor chain of size $2$ in the ordering, $\bar a$ is definable by the $\Pinf{2}$ formula stating that $a_1$ and $a_2$ are successors and the tuple has no immediate predecessors or successors, and for each $c$ and $\Gamma$, $L_c+2+P_{\Gamma}$ has different isomorphism type.
    Furthermore, as $\SR_p(P_\Gamma)\geq 3$ we get that
    \[\SR(L_c+2+P_\Gamma,\bar{a},\bar{p})=\SR(L_c+2+P_\Gamma,\bar{p})\]
    for any (potentially empty) set of parameters $\bar{p}$ (see e.g.~\cite[Table 1]{MBook}).
	It follows that
    \[\SR(L_c+2+P_\Gamma)=\SR(L_c+2+P_\Gamma,\bar{a})=\SR(P_\Gamma)\] and
    \[\SR_p(L_c+2+P_\Gamma)=\SR_p(L_c+2+P_\Gamma,\bar{a})=\SR_p(P_\Gamma),\]
    where the second equality in each case follows from \cite[Lemma 11]{GM23}.
    This means that the family $L_c+2+P_\Gamma$ gives continuum many non-isomorphic linear orderings that all have the same parameterized and unparameterized Scott rank.

    We now show that $P_\Gamma$ has an unstable $\lambda$-sequence over a (potentially empty) set of parameters $\bar{p}$ if and only if $L_c+2+P_\Gamma$ does.
    Say that $P_\Gamma$ has such a sequence $\bar{y}_i$ over $\bar{p}$ and fundamental sequence $\alpha_i\to\lambda$.
    Let $\iota$ be the natural embedding of $P_\Gamma$ into $L_c+2+P_\Gamma$.
    Consider $\iota(\bar{y}_{i})$ over $\iota(\bar{p})$.
    We will modify this sequence to produce our desired unstable $\lambda$-sequence over parameters. The following two claims will be useful.

    \begin{claim}\label{claim:forwardEmbed}
    If $\bar{a}\equiv_{\alpha_i}\bar{b}$ in $P_\Gamma$, then $\iota(\bar{a})\equiv_{\alpha_i}\iota(\bar{b})$ in $L_c+2+P_\Gamma$.
    \end{claim}

    \begin{proof}
        By the definition of $\iota$,
        \[(P_\Gamma,\iota(\bar{a}))\cong (P_\Gamma,\bar{a})\equiv_{\alpha_i} (P_\Gamma,\bar{b})\cong(P_\Gamma,\iota(\bar{b})).\]
        Note further that $L_c+2\cong L_c+2$, so  $L_c+2\equiv_{\alpha} L_c+2$.
        In total, this gives that 
        \[L_c+2+(P_\Gamma,\iota(\bar{a}),\iota(\bar{p}))\equiv_{\alpha_i} L_c+2+(P_\Gamma,\iota(\bar b),\iota(\bar{p})),\]
    by playing independently on each summand.
    \end{proof}

    Note that \cref{claim:forwardEmbed} gives that $\iota(\bar{y}_{i}),\iota(\bar{p})\equiv_{\alpha_i} \iota(\bar{y}_{i+1}),\iota(\bar{p})$.

    \begin{claim}\label{claim:backwardEmbed}
    If $\iota(\bar{a})\equiv_{\alpha+2}\iota(\bar b)$ in $L_c+2+P_\Gamma$, then $\bar a \equiv_\alpha \bar b$ in $P_\Gamma$.
    \end{claim}

    \begin{proof}
        Say that $\bar{a}\not\equiv_{\alpha} \bar{b}$.
        There must be some $\varphi$ of quantifier rank $\alpha$ true of $\bar{a}$ yet false of $\bar{b}$.
        Let $\hat{\varphi}$ be the same as $\varphi$, but with the quantifiers relativized to the elements to the right of a 1-block of size 2.
        This increases the complexity of $\varphi$ by at most 2 quantifiers. In particular, $\iota(\bar{a})\not\equiv_{\alpha+2}\iota(\bar{b})$, as $\hat{\varphi}$ is true of $\iota(\bar{a})$ yet false of $\iota(\bar{b})$.
    \end{proof}
		Now fix some $i$, then $\bar{y}_{i},\bar{p}\equiv_{\alpha_{i}} \bar{y}_{i+1},\bar{p}$, $\bar{y}_{i},\bar{p}\not\equiv_{\alpha_{i+1}} \bar{y}_{i+1},\bar{p}$. By \cref{claim:backwardEmbed} $\iota(\bar{y}_{i}),\iota(\bar{p})\not\equiv_{\alpha_{i+2}} \iota(\bar{y}_{i+1}),\iota(\bar{p})$. By \cref{claim:forwardEmbed} and the defining property of a $\lambda$-sequence we have that
    \[\iota(\bar{y}_{i}),\iota(\bar{p})\equiv_{\alpha_i} \iota(\bar{y}_{i+2}),\iota(\bar{p}) \text{ and } \iota(\bar{y}_{i}),\iota(\bar{p})\not\equiv_{\alpha_{i+2}} \iota(\bar{y}_{i+2}),\iota(\bar{p}).\]
    In particular, $\iota(\bar{y}_{2i})$ is an unstable $\lambda$-sequence over $\iota(\bar{p})$ and the fundamental sequence $\alpha_{2i}\to\lambda$, as desired.
    
		In~\cite[Corollary 2.11]{GR23} it was shown that a structure has an unstable $\lambda$-sequence if and only if there are a fundamental sequence $\alpha_i\to \lambda$ and complete $\Pinf{\alpha_i}$-types $p_i(\bar x)$ realized in the structure, but there is no single tuple realizing all of them. So say that $L_c+2+P_\Gamma$ has an unstable $\lambda$-sequence, i.e., there is a sequence of realized types $p_i(\bar x)$ as above but no single element realizing all of them. Since the Scott rank of $L_c$ is bounded below $\lambda$ and the elements in the $2$-block are definable by simple formulas, for cofinitely many $i$, $p_i$ must be realized by tuples containing elements from $P_\Gamma$.
        Furthermore, for the same reason, we may assume that for cofinitely many $i$, the realizers of $p_i$ are tuples from the $P_\Gamma$ copy in $L_c+2+P_\Gamma$. Thus, we may assume without loss of generality that all $p_i$ are realized by elements $\iota(\bar z_i)$, $\bar z_i\in P_\Gamma^{<\omega}$. Consider the $\Pinf{\alpha_i}$-types $q_i$ of $\bar z_i$ and suppose there is a single $\bar z$ satisfying all of them. Then $\bar z\equiv_{\alpha_i}\bar z_i$ for every $i$, and by \cref{claim:forwardEmbed}, $\iota(\bar z)\equiv_{\alpha_i} \iota(\bar z_i)$, so in particular $L_c+2+P_\Gamma\models p_i(\iota(\bar z)) $ for all $i$. But this contradicts $p_i$ not having a single realizer in $L_c+2+P_\Gamma$, so such a $\bar z$ cannot exist for the sequence of types $q_i$. Thus, by~\cite[Corollary 2.11]{GR23}, $P_\Gamma$ has an unstable $\lambda$-sequence.
\end{proof}

\begin{theorem}
    Let $LO_k$ be the theory of linear ordering with $2\leq k<\aleph_0$ colors. 
    \begin{enumerate}
        \item $I(LO_k,\Pinf{1}) = 1$,
        \item $I(LO_k,\Sinf{2})=0$,
        \item $I(LO_k,\Pinf{2}) = \aleph_0$,
	\item $I(LO_k,\Sinf{3}) = \aleph_0$,
	\item $I(LO_k,\text{d-}\Sinf{i}) = \aleph_0$ if $i\leq 2$,
	\item $I(LO_k,\Gamma) = 2^{\aleph_0}$ otherwise.
    \end{enumerate}
The orderings of complexity $\Sinf{3}$ or below are among the finite sums of the following basic linear orderings:
\[\mathbb{B}:= \{1_i\}_{i\in k} \cup \{\eta_i\}_{i\in k} \cup \{Sh_c(A)\}_{A\subseteq k}.\]
\end{theorem}

\begin{proof}
    By \cref{thm:nocolors}, for every $\Gamma\geq\Pinf{4}$, there are continuum many linear orderings (and therefore linear orderings with finitely many colors) of each of these complexities.
    We focus on the cases where $\Gamma\leq \Sinf{4}$.

	Consider the colored orderings of Scott complexity strictly less than $\Pinf{2}$. These must be finite as there is no countably infinite structure with a $\Sinf{2}$ Scott sentence (see e.g. \cite{AGNHTT} Theorem 5.1).

	The empty colored linear ordering has a $\Pinf{1}$ Scott sentence, and it is straightforward to confirm that for any other colored linear ordering $\mc{F}$, $\emptyset\geq_1 \mc{F}$ and so $\mc F$ cannot have a $\Pinf{1}$ Scott sentence.
    This gives exactly one linear ordering of complexity $\Pinf{1}$.
    Every other finite, finitely colored linear ordering has a $\text{d-}\Sinf{1}$ Scott sentence (as is always the case for finite structures in a finite language).
    There are countably many of these, so there are $\aleph_0$ many structures of complexity $\text{d-}\Sinf{1}$.

    This leaves the cases of $\Pinf{2},\Pinf{3},\Sinf{3},\Sinf{4},\text{d-}\Sinf{2}$ and $\text{d-}\Sinf{3}$.

    We begin with the case of $\Pinf{2}, \Sinf{3}$ and $\text{d-}\Sinf{2}$.
    Colored linear orderings of complexity $\Sinf{3}$ and $\text{d-}\Sinf{2}$ must be finite sums of colored linear orderings of complexity $\Pinf{2}$ by Lemma \ref{lem:colorHelper}.
    Therefore, to show that there are at most countably many linear orderings of each of these complexities, it is enough to show that there are only countably many linear orderings of complexity $\Pinf{2}$.
    To show this, it is enough to demonstrate that there is a countable 2-universal class for this theory (see \cite{GR23} Lemma 2.2).
    That is, we fix a countable class of structures $\mathbb{C}$ and for each finitely colored linear ordering $\mc{L}$ we find $\C\in\mathbb{C}$ with $\mc{L}\leq_2\C$.
    We claim that the class $\mathbb{C}=\mathbb{B}^*$ of finite sums of linear orderings in $\mathbb{B}$ is 2-universal.
    We will show this by induction on the number of colors, with the base case $k=1$ following from the uncolored analysis done in \cite{GR23} Lemma 2.4.
    
    Now let $k>1$.
    Fix a $k$ colored linear ordering $\mc{L}$.
    We construct some $\mc{M}\in\mathbb{B}^*$ with $\mc{M}\geq_2 \mc{L}$.
    
    First, consider the case that every sequence of colors $\bar{c}=c_1,\cdots,c_n$ in $k^{<\omega}$ is realized in $\mc{L}$ by some $\bar{a}_{\bar{c}}=a_1<\cdots<a_n$.
    In this case we claim that $\mc{L}\leq_2 Sh_c(k)$.
		Say the $\forall$-player plays the $\bar{c}$-colored tuple $\bar{b}$.
    The $\exists$-player then responds by playing $\bar{a}_{\bar{c}}$.
    To confirm that this is a winning move, we need to show for all $0\leq i\leq n$ that $(b_i,b_{i+1})_{Sh_c(k)}\leq_1(a_i,a_{i+1})_{\mc{L}}$.
    Note that it follows from \cref{cor:self-sim} that $(b_i,b_{i+1})_{Sh_c(k)}\cong Sh_c(k)$.
    As every possible increasing sequence of colors is realized in $Sh_c(k)$, the claim follows by \cref{lem:colorHelper}.
    
    Therefore, we may assume that there is some sequence of colors $\bar{c}=c_1,\cdots,c_n$ in $k^{<\omega}$ that is not realized by any $\bar{a}\in\mc{L}$.
    We will proceed with induction on $n$ to demonstrate that in this case $\mc{L}$ can be written as a finite sum of colored linear orderings where fewer than $k$ colors are used.
    For the base case $n=1$, the claim follows immediately as this means that $\mc{L}$ is missing a color.
    Now let $n>1$.
    For the color $c_1$, let $\mc{L}_{>c_1}=\{x\in\mc{L}\mid \exists y \; P_{c_1}(y) \land x>y\}.$
    $\mc{L}_{>c_1}$ must omit the sequence of colors $c_2,\cdots,c_{n}$ or else, by definition, the realizing tuple can be extended to a tuple that realizes $c_1,\cdots,c_n$.
    By induction, $\mc{L}_{>c_1}$ can be written as the finite sum of colored linear orderings where fewer than $k$ colors are used.
    Because $\mc{L}_{>c_1}$ is upward closed, we can write $\mc{L}=P+\mc{L}_{>c_1}$ for some colored linear ordering $P$ so that no element in $P$ is greater than any element of color $c_1$.
    This means that $P$ either omits the color $c_1$ or can be written $P=P'+1_c$, where $P'$ omits the color $c_1$.
    Either way, $P$ can also be written as the finite sum of colored linear orderings where fewer than $k$ colors are used.
    As $\mc{L}=P+\mc{L}_{>c_1}$, putting the sums together gives a way of writing $\mc{L}$ as a finite sum of colored linear orderings where less than $k$ colors are used.
    This ends this sub-induction and demonstrates that we may assume that $\mc{L}=A_1+\cdots+A_n$ where each $A_i$ has less than $k$ colors used.
    By the induction on the number of colors, we can assign $A_i'\in \mathbb{B}^*$ to each $A_i$ with $A_i'\geq_2 A_i$.
    Let $\mc{M}=A_1'+\cdots+A_n'$.
    It follows that $\mc{M}\in \mathbb{B}^*$ and $\mc{L}\leq_2 \mc{M}$.

    Overall, this demonstrates that there are at most countably many colored linear orderings of Scott complexity at most $\Sinf{3}$, as they all must be contained in $\mathbb{B}^*$.
    We now observe that there are infinitely many colored linear orderings of Scott complexities  $\Pinf{2}, \Sinf{3}$ and $\text{d-}\Sinf{2}$.
    
    For $\Pinf{2}$ consider $(\eta_0+\eta_1)\cdot n$.
    Each copy of $\eta_i$ is existentially definable, and the automorphism orbits within the copies are quantifier-free definable.
    Putting this together gives that the automorphism orbits in $(\eta_0+\eta_1)\cdot n$ are existentially definable (without parameters), so there is a $\Pinf{2}$ Scott sentence.
    
    For $\text{d-}\Sinf{2}$ consider $(\eta+n+\eta)_0$.
    Each of the $\eta+n+\eta$ is $\text{d-}\Sinf{2}$ as a linear ordering (see e.g. Theorem 2.6 of \cite{GR23}), so $(\eta+n+\eta)_0$ is $\text{d-}\Sinf{2}$ as a colored linear ordering.
    
    For $\Sinf{3}$ consider $\eta_0+1_0+Sh_c(0,1)+1_1+(\eta_0+\eta_1)\cdot n$.
    The same argument as in \cref{thm:allComplexities} shows that the "$1_0$" has no $\Sinf{2}$ definable orbit.
    Furthermore, as $\eta_0$, $Sh_c(0,1)$ and $(\eta_0+\eta_1)\cdot n$ have $\Pinf{2}$ Scott sentences, the orbits are $\Sinf{1}$ definable over parameters, and so all of these orderings have Scott complexity $\Sinf{3}$.
   
    We continue with the case of $\Pinf{3}$.
    Given $A\subseteq \omega$, let $\omega_A$ be the 2-colored linear ordering of order type $\omega$ with $P_0(x)$ if and only if $x$ is the $i^{th}$ element and $i\in A$.
    This gives continuum many $2$-colored linear orderings with $\Pinf{3}$ Scott sentences.
    These orderings cannot have simpler Scott sentences because $\omega_A$ is not among $\mathbb{B}^*$, and so we get continuum many orderings of complexity $\Pinf{3}$, as desired.

    Now, to show that there are continuum many orderings of complexity $\text{d-}\Sinf{3}$ consider the colored linear orderings $\omega_A\cdot 2$.
		If $A$ is infinite and coinfinite, then we claim that $\omega_A\cdot 2\leq_3\omega_A$ and thus $\omega_A\cdot 2$ cannot have a $\Pinf{3}$ Scott sentence.
		Given a tuple with greatest element $a \in\omega_A$ played by the $\forall$-player on their first move, the $\exists$-player responds along the natural initial embedding $\iota:\omega_A\to \omega_A\cdot 2$. It is enough to show that $\omega_{A\geq a}\leq_2 (\omega_A\cdot 2)_{\geq \iota(a)}$ by Item (1) of \cref{lem:colorHelper}.
		Consider a play $d_1<\cdots<d_n$ from the $\forall$-player on $((\omega_A)\cdot 2)_{\geq \iota(a)}=(\omega_A)_{\geq \iota(a)}+\omega_A$.
	For the $d_1 <\cdots <d_k$ that are played in $(\omega_{A})_{\geq \iota(a) }$, the $\exists$-player may play along $\iota^{-1}$.
	For $d_{k+1}<\cdots<d_n$, the $\exists$-player needs to find $c_{k+1}<\cdots<c_n$ in $\omega_A$  with $(d_{i},d_{i+1})\leq_1 (c_{i},c_{i+1})$ for $k\leq i<n$. 
    By Item 2 of \cref{lem:colorHelper} it is enough to find $d_{k+1}<\cdots<d_n$ where $(d_{i},d_{i+1})$ takes an embedding from $(c_{i},c_{i+1})$ for each $k\leq i<n$.
    This is clearly possible as $A$ is infinite and co-infinite.

    Therefore, $\omega_A\cdot 2\leq_3\omega_A$ and $\omega_A\cdot 2$ has no $\Pinf{3}$ Scott sentence.
    However, there is a $\Pinf{3}$ Scott sentence over the first element of the second $\omega_A$ as a parameter.
    This is true because each $\omega_A$ has a $\Pinf{3}$ Scott sentence.
    Furthermore, this element has a $\Pinf{2}$ isomorphism orbit given by saying it has elements before it but no direct predecessor.
    This means that each $\omega_A\cdot 2$ has Scott complexity $\text{d-}\Sinf{3}$ when $A$ is infinite and coinfinite, and thus there are continuum many colored orderings of this Scott complexity.
    Finally,  consider the case of $\Sinf{4}$ and the colored linear orderings $(\eta+1+2\cdot\eta)_0+\omega_A.$
    The same argument as in Theorem 3.3 of \cite{GR23} shows that the "1" in the $\eta+1+2\cdot\eta$ has no $\Sinf{3}$ definable automorphism orbit.
    Over the parameters of the "1" and the first element of $\omega_A$, each interval has a $\Pinf{3}$ Scott sentence.
    Therefore, each instance of $(\eta+1+2\cdot\eta)_0+\omega_A$ must have Scott complexity $\Sinf{4}$.

    As all cases have been demonstrated, the theorem follows.
\end{proof}

Note that the above analysis fails to fully apply to linear orderings with infinitely many colors.
For example, one can color any infinite structure with infinitely many colors so that every color occurs exactly once. Such a structure has Scott complexity $\Pinf{2}$ because the automorphism orbit of every tuple is simply given by the colors of the tuple's elements.

\section{From colored linear orderings to models of $\PA$}\label{sec:coloredtomodels}
In~\cite{MR23}, Montalbán and Rossegger showed that the class of linear orderings is reducible via $\Dinf{\omega}$ bi-interpretability to the class of models of $T$ for a given completion $T$ of Peano arithmetic. Given a linear ordering $\mc{L}$ and the prime model $\N$ of $T$, they interpreted $\mc{L}$ in a model $\N_{\mc{L}}$, the so-called \define{$\mc{L}$-canonical extension} of $\N$. These extensions, due to Gaifman~\cite{gaifman1976}, can be viewed as classical Ehrenfeucht-Mostowski models where all elements in the set of order indiscernibles $X\subseteq N_\mc{L}$ generating the model have a so-called \define{minimal type}. For a thorough treatment of minimal types and $\mc{L}$-canonical extensions, we refer to~\cite{KS}. The main property of a minimal type for us is that if $\N(a)$ is such that $a$ has type $p$, then the only proper elementary submodel of $\N(a)$ is $\N$. Roughly speaking, the reduction in~\cite{MR23} is obtained by iteratively extending $\N$ via elements of type $p$, so that in $\N_\mc{L}$ every non-standard element is in the Skolem hull of an element of type $p$. 

Our reduction from colored linear orderings to models of $T$ is similar to the one given in~\cite{MR23} except for the key difference that instead of adjoining elements of the same fixed type $p$, given a finite ordered sequence $x_0<\dots< x_n$ where $x_i$ has color $c_i$, we build $\N(x_0,\dots,x_n)$ using minimal types $p_i$ where $p_i=p_j$ if and only if $c_i=c_j$. To do this, we need a countable sequence of minimal types, and thankfully by a clever application of Ramsey's theorem, one can show that there are $2^{\aleph_0}$ such types~\cite[Theorem 3.1.4]{KS}. 
Given $\mc{L}$ and an injection $c_i\mapsto p_i$ from colors to minimal types consistent with $T$, the model $\N_\mc{L}$ is obtained as the direct limit of the models $\N(x_0,\dots,x_n)$ where $x_0<\dots<x_n$ is an ordered sequence from $\mc L$.

The construction of $\N_\mc{L}$ is discussed in detail in~\cite{KS} and~\cite{MR23}. We call $\N_\mc{L}$ the \define{canonical colored $\mc{L}$-extension} and the class consisting of the canonical colored $\mc{L}$-extensions for all $\mc{L}$ as the class of \define{colored extensions}. We will now explain how the map $\mc{L}\to \N_\mc{L}$ induces a reduction via $\Dinf{\omega}$ bi-interpretability of the class of linear orderings to the class of colored extensions of $T$ and why this reduction preserves Scott complexity. Our reasoning is quite different from the reasoning in~\cite{MR23} that shows that the corresponding reduction from linear orderings to models of $T$ preserves Scott rank. We begin by recalling relevant definitions.

\subsection{Infinitary bi-interpretability and the canonical $\omega$-jump of models of $\PA$}
\begin{definition}[\cite{HTMM}]\label{def:infint}
  A structure $\A=(A,P_0^\A,\dots)$ (where $P_i^\A\subseteq A^{a(i)}$) is
  \define{infinitarily interpretable} in $\B$ if there are relations
  $Dom_\A^\B,\sim,R_0,R_1,\dots$, each definable in $L_{\omega_1\omega}$ without
  parameters in the language of $\B$ such that
  \begin{enumerate}
    \tightlist
  \item $Dom_\A^\B\subseteq \B^{<\omega}$,
  \item $\sim$ is an equivalence relation on $Dom^\B_\A$,
  \item $R_i\subseteq (Dom^\B_\A)^{a(i)}$ is closed under $\sim$,
  \end{enumerate}
  and there exists a function $f_\A^\B: Dom_\A^\B\to \A$ which induces an
  isomorphism:
  \[ f_\A^\B: \A^\B=(Dom_\A^\B, R_0,R_1,...){/}{\sim}\cong \A.\]
  We say that $\A$ is \define{$\Dinf{\alpha}$ interpretable} in $\B$ if 
  the above relations are both $\Sinf{\alpha}$ and $\Pinf{\alpha}$ definable in
  $\B$.
\end{definition}
\begin{remark}
	The cautious reader might have noted that in the above definition we use relations of unbounded arity.	We say that $R\subseteq (\B^{<\omega})^k$ is $\Sinf{\alpha}$-definable if there is a sequence of $\Sinf{\alpha}$ formulas $(\phi_{n_1,\dots,n_k})_{(n_1,\dots,n_k)\in\omega^k}$ such that for $(\ba_1,\dots \ba_k)\in \B^{n_1}\times \dots\times \B^{n_k}$ $(\ba_1,\dots,\ba_k) \in R$ if and only if $\B\models \phi_{n_1,\dots, n_k}(\ba_1,\dots,\ba_k)$. This allows us to decouple definability from structural constraints. For example, even though one cannot define any non-trivial relation on $\mathbb Q^{n}$, one can code any subset of $\omega$ in $\mathbb Q^{<\omega}$.
\end{remark}
\begin{definition}[\cite{HTMM}]\label{def:infbiint}
  Two structures $\A$ and $\B$ are \define{infinitarily bi-interpretable} if
  there are interpretations of each structure in the other such that the
  compositions
  \[ f_\B^\A\circ \tilde f_\A^\B: Dom_\B^{(Dom^\B_\A)} \to \B \text{ and
  } f_\A^\B\circ \tilde f_\B^\A: Dom_\A^{(Dom^\A_\B)} \to \A\] 
  are $L_{\omega_1\omega}$ definable in $\B$, respectively, $\A$. 
  If $\A$ and $\B$ are $\Dinf{\alpha}$ interpretable in each other and the
  associated compositions are $\Dinf{\alpha}$ definable, then we say that $\A$
and $\B$ are \define{$\Dinf{\alpha}$ bi-interpretable}.\end{definition}
The following definition is a generalization of reducibility via effective
bi-interpretability. See~\cite{montalban2021} for a thorough treatment of
effective bi-interpretability.
\begin{definition}
  A class of structures $\mathfrak C$ is \define{reducible via infinitary
  bi-interpretability} to a class of structures $\mathfrak D$ if there are
  infinitary formulas defining domains, relations, and isomorphisms of an
  infinitary bi-interpretation so that every structure in $\mathfrak C$ is
  bi-interpretable with a structure in $\mathfrak D$ using this
  bi-interpretation. If all the formulas are $\Dinf{\alpha}$ then we say that
  $\mathfrak C$ is \define{reducible via $\Dinf{\alpha}$ 
  bi-interpretability to $\mathfrak D$}.
\end{definition}

We will describe below how the map $\mc{L}\to \N_\mc{L}$ gives a reduction via infinitary bi-interpretation from colored linear orderings to models of $T$ for a given completion $T$ of $\PA$. 
On first sight this map gives rise to a reduction via $\Dinf{\omega}$ bi-interpretability, but upon closer look one gets something much stronger. Namely, we get that $\mc{L}$ is $\Dinf{1}$ bi-interpretable with the canonical $\omega$-jump of $\N_\mc{L}$. As is shown in~\cite[Proposition 16]{MR23}, for a model $\M$ of $\PA$, the canonical $\omega$-jump of $\M$ is the structure $\M_{(\omega)}=(\M,(S_n)_{n\in\omega})$ where $(S_n)_{n\in\omega}$ is a listing of the types realized in $\M$. The drawback of this definition is that it is non-uniform. 
Any completion of $\PA$ has uncountably many consistent types, and thus, interpreting $\M_{(\omega)}$ in $\M$ may require a different choice of formulas than interpreting $\N_{(\omega)}$ in $\N$ if $\N\not\cong\M$. Thankfully, in colored extensions all types possibly realized are determined by the Skolem terms over the types $(p_i)_{i\in\omega}$ corresponding to the colors, and thus there is only a countable set of types to consider. Thus, for a canonical colored $\mc{L}$ extension we define its canonical $\omega$-jump as follows. Fix a finite ascending sequence $x_0<\dots <x_m=\bar x$ where $x_i$ has color $c_i$.
Let $(r_{(\bar n,\bar c)})_{(\bar n,\bar c)\in\omega^{<\omega}\times \omega^{<\omega}}$ be a listing of types such that $r_{\bar n,\bar c}$ is the type of $t_{n_0}(\bar x)\dots t_{n_{|\bar n|}}(\bar x)$ where $t_{n_i}$ is the $n_i$th Skolem term in $\N(\bar x)$ where $\N(\bar x)$ is obtained by adjoining $p_{c_0}\times\dots\times p_{c_{|\bar c|}}$. Then the \define{canonical $\omega$-jump} of $\N_\mc{L}$ is the structure
\[ (\N_\mc{L})_{(\omega)}=(\N_\mc{L},(\{ \bar a: r_{\bar n,\bar c}(\bar a)\})_{(\bar n,\bar c)\in\omega^{<\omega} \times \omega^{<\omega}}).\]

This definition has the advantage that it is independent of $\mc{L}$ and thus uniform in the colored extensions.
\begin{theorem}\label{thm:biint}
	For $T$ any completion of $\PA$, the class of colored linear orderings is reducible via $\Dinf{1}$ bi-interpretability to the canonical $\omega$-jumps of colored extensions of $T$.
\end{theorem}
Reductions such as the one in \cref{thm:biint} were thoroughly analyzed in~\cite{MR23} and this together with a back-and-forth analysis was used to connect the Scott ranks of $\mc{L}$ and $\N_\mc{L}$ (for $\mc{L}$ a linear ordering). However, the authors did not notice that using the uniformity of the interpretations allows one to directly relate Scott sentence complexities. This is the route we are taking here.
\subsection{Interpreting $(\mc{N}_\mc{L})_{(\omega)}$ in $\mc{L}$}
In order to interpret models of $\PA$ in linear orders, we rely heavily on the feature that the arity of our relations may be unbounded.

Every element in $(\mc{N}_{\mc{L}})_{(\omega)}$ is of the form $t(\bar x)$ where $t$ is a Skolem term consistent with $T$ and $\bar x$ is a tuple in $\mc{L}$. Fix a bijection $f$ of these Skolem terms and $\mathbb N$. We will now formally define $Dom_{\N_\mc{L}}^\mc{L}$ and $\sim$.
Elements of $Dom_{\N_\mc{L}}^\mc{L}$ will be tuples of the form $\bar{a}d\bar{b}$ where $\bar a$ and $\bar b$ are ascending sequences in $\mathcal{L}$ and $d=a_{|\bar a|}$, i.e., the last element of $\bar a$. The idea is that $\bar{a}d\bar{b}$ corresponds to the element $t_{f(|\bar b|)}(\bar a)$; the element $d$ only acts as a buffer between $\bar{a}$ and $\bar{b}$ to make it clear where one tuple ends and the other begins. Given two ordered tuples $\bar a$ and $\bar a'$, let $\sigma(\bar a\bar a')$ be the reordering of these tuples in ascending order without repetitions. We can thus define $\sim$ informally, assuming that $c(x)$ is the color of an element $x$ as
\[\bar{a}d\bar{b}\sim \bar{a}'d'\bar{b}' \iff \bar n=\sigma(\bar a\bar a') \land p_{c(n_1)}\times \dots \times p_{c(n_{|\bar{n}|})}\models t_{f(|\bar{b}|)}(\bar a)=t_{f(|\bar{b}'|)}(\bar a')\]
To see that the right-hand side is $\Sinf{1}$ definable, consider tuples $\bar x, \bar x'$ of length $n$, respectively, $m$. First we need to decode the tuples $\bar x$, $\bar x'$ to find the delimiters $d$, $d'$. This can be done using a finite disjunction. Then, we reorder the resulting tuples $\bar a$, $\bar a'$ into the tuple $\bar n$ (this is again a finite disjunction). Now, we have an infinite disjunction having as disjuncts the color arrangements so that $t_{f(|\bar b|)}(\bar a)=t_{f(|\bar b'|)}(\bar a')$ is in the type.
Addition, multiplication, $\underbar 0$ and $\underbar 1$ are straightforward to define, and so we get an interpretation of $\N_\mc{L}$ in $\mc{L}$. To interpret $(\mc{N}_\mc{L})_{(\omega)}$ it remains to interpret the types $r_{\bar n,\bar c}$. It is easy to obtain a $\Sinf{1}$ definition of these types, and, as a tuple can have at most one complete type, to obtain a $\Sinf{1}$ definition of the cosets, one can simply take the disjunction over all other types as a definition. 
The resulting interpretation is independent of $\mc{L}$.
\subsection{Interpreting $\mc{L}$ in $(\mc{N}_\mc{L})_{(\omega)}$ and proving \cref{thm:biint}}
The key insight to obtain an interpretation of $\mc{L}$ in $(\mc{N}_\mc{L})_{(\omega)}$ is that if we consider the subset of $\mc N_\mc L$ made up of elements of any of the types $p_i$ used in the construction, then this set has order type $ot(\mc{L})$ when equipped with the canonical ordering of $\mc N_\mc L$. This is explained in full detail in~\cite{MR23}, and also in~\cite[Section 3.3]{KS}, see in particular~\cite[Theorem 3.3.5]{KS}.
This train of thought leads us to define the interpretation as 
\[ Dom_\mc{L}^{\mc{N}_\mc{L}} = \{x\mid \bigvvee_i \tp(x)=p_i\},\, \sim~=~=,\,  <_\mc{L}=<_{\mc{N}_\mc{L}},\, P_k=\{x\mid  \tp(x)=p_k\}.\]
At first sight, obtaining a $\Sinf{1}$ definition of $\overline{Dom_{\mc{L}}^{\N_\mc{L}}}$ that is independent of $\mc{L}$ seems to be difficult. However, we can again use the fact that among colored extensions we can only find a countable set of types realized. So, let $(r_i)$ be the sublisting of these $1$-types so that $r_i\neq p_j$ for all $j$ and $i$. Then 
\[ \overline{Dom_\mc{L}^{\mc{N}_\mc{L}}}=\{ x\mid \bigvvee_{i} \tp(x)=r_i\}.\]

To finish the proof of \cref{thm:biint}, it remains to show that the isomorphisms $f_\mc{L}^{\N_\mc{L}}\circ \tilde{f}_{\N_\mc{L}}^\mc{L}$ and $f_{\N_\mc{L}}^\mc{L}\circ \tilde{f}_\mc{L}^{\N_\mc{L}}$ are $\Dinf{1}$-definable in $\mc{L}$, respectively, $(\N_\mc{L})_{(\omega)}$. For the former, note that the equivalence classes in $dom(f_\mc{L}^{\N_\mc{L}}\circ \tilde{f}_{\N_\mc{L}}^\mc{L})=Dom_\mathcal{L}^{(Dom_{\N_{\mc L}}^\mc{L})}$ have canonical elements of the form $ad\bar b$ where $|\bar b|$ is the code for the identity function. Thus, we get
\[f_\mc{L}^{\N_\mc{L}}\circ \tilde{f}_{\N_\mc{L}}(\bar ad\bar b)= y \iff\exists x_1,\dots x_{|\bar b|}\ \bar ad\bar b\sim yy\bar x.\]
Again using the fact that we can decode tuples into the form $\bar a d\bar b$ and that may use relations of unbounded arity, both this and the negation is easily seen to be $\Sinf{1}$.
 Elements in the domain of the latter are of the form $\bar a c \bar b$ where these tuples are tuples from $\N_\mc{L}$ with every element of type $p_i$ for some $i$. According to our definitions these are to be interpreted as $t_{f(|\bar b|)}(\bar a)$, and the map $\bar a c\bar b'\mapsto t_{f(|\bar b|)}(\bar a)$ defines the isomorphism. The graphs of both functions are trivially $\Dinf{1}$-definable as required.
\subsection{Using the bi-interpretation for Scott complexity}
One can show that having a reduction via $\Dinf{1}$ bi-interpretability from a class $\mathfrak C$ into a class $\mathfrak D$ allows one to define $\mathfrak C$ in $\mathfrak D$~\cite[Lemma VI.29]{montalban2021}. Applying this lemma in our context shows that the class of canonical $\omega$-jumps of colored extensions is $\Pinf{2}$ definable in the class of canonical $\omega$-jumps of $T$.

Using the easily seen fact that the relations $r_{m,n}$ are $\Pinf{\omega}$ definable in $\N_\mc{L}$, we then obtain the following lemma.

\begin{lemma}
    The colored extensions of $T$ are $\Pinf{\omega+2}$ definable in the language of $\PA$.
\end{lemma}

This lemma allows us to prove the following.

\begin{proposition}\label{prop:ComplexityTransfer}
	For $\mc L$ an infinite colored linear ordering, if $\mc{L}$ has Scott complexity $\Gamma^{\mathrm{in}}_\alpha\geq\Pinf{2}$ then $\mc{N}_{\mc{L}}$ has Scott complexity $\Gamma^{\mathrm{in}}_{\omega+\alpha}$.
\end{proposition}

\begin{proof}
    Let $\theta$ be the $\Pinf{\omega+2}$ formula that defines the image of $\mc{L}\mapsto\mc{N}_\mc{L}$.
		Let $\varphi$ be a $\Gamma^{\mathrm{in}}_\alpha$ Scott sentence for $\mc{L}$.
		Our reduction via bi-interpretability from colored linear orderings onto the canonical $\omega$-jumps of colored extensions provides a continuous reduction of either class in the other, and hence combining the classical pull-back theorem~\cite[Theorem XI.7]{MBook} with the fact that the $\omega$-jump is $\Pinf{\omega}$ interpretable we get that there is a $\Gamma^{\mathrm{in}}_{\omega+\alpha}$ formula in the language of Peano arithmetic such that
    \[\mc{L}\models\varphi \iff \mc{N}_\mc{L}\models\varphi_*.\]
    Let $\mc{M}\models\varphi_*\land\theta$.
    Because $\mc{M}\models\theta$, $\mc{M}\cong\mc{N}_{\mc{L}'}$ for some $\mc{L}'$.
    As $\mc{N}_{\mc{L}'}\models \varphi_*$, $\mc{L}'\models \varphi$.
    Because $\varphi$ is a Scott sentence $\mc{L}'\cong\mc{L}$ and so $\mc{M}\cong \mc{N}_{\mc{L}'}\cong \mc{N}_{\mc{L}}$.
    This means that $\varphi_*\land\theta$ is a Scott sentence for $\mc{N}_\mc{L}$.
		As the minimal Scott complexity of infinite structures is $\Pinf{2}$, $\varphi_*\land\theta$ is of complexity $\Gamma^{\mathrm{in}}_{\omega+\alpha}$.

		To see that there is no simpler Scott sentence, suppose for the sake of contradiction that $\psi$ is a Scott sentence for $\mc{N}_\mc{L}$ of complexity $\Theta^{\mathrm{in}}_{\omega+\beta}$ with $\Theta_{\omega+\beta}^{\mathrm{in}}<\Gamma^{\mathrm{in}}_{\omega+\alpha}$. Let us observe that by Theorem \cite[Theorem 8]{MR23} or Proposition \ref{cor_noPiomega} in this paper, $\beta>0$. This implies that there is a $\Theta^{\mathrm{in}}_\beta$-formula $\psi'$ in the language of the canonical $\omega$-jump of colored linear orderings such that 
		\[ \N_\mc{L} \models \psi\iff (\N_\mc{L})_{(\omega)}\models \psi'\]
	and again by the pull-back theorem there is a $\Theta^{\mathrm{in}}_\beta$-formula $\psi^*$ in the language of linear orderings such that 
	\[ (\N_\mc{L})_{(\omega)}\models \psi'\iff \mc{L}\models \psi^*.\]
	But as $\psi$ is a Scott sentence, so is $\psi'$ and also $\psi^*$ is a Scott sentence for $\mc{L}$ of complexity $\Theta^{\mathrm{in}}_\beta<\Gamma^{\mathrm{in}}_\beta$, a contradiction.
\end{proof}

\begin{theorem}\label{thm:transferedcomplexities}
    There are continuum many models of Peano arithmetic of Scott complexity 
    \begin{enumerate}
        \item $\Sinf{\omega+\alpha}$ for $\alpha>2$ and not a limit.
        \item $\dSinf{\omega+\alpha}$ for $\alpha>1$ not a limit.
        \item $\Pinf{\omega+\alpha}$ for $\alpha>1$.
    \end{enumerate}
    
\end{theorem}

\begin{proof}
    The existence of such models follows from \cref{thm:allComplexities} and \cref{prop:ComplexityTransfer}.
    To get continuum many, one can pick among the continuum many different minimal types to execute the construction $\mc{L}\mapsto\mc{N}_\mc{L}$ relativized to an oracle that can calculate the relevant types.
\end{proof}

We can also extract some counterexamples to conjectures made by \L{}e\l{}yk at the Computable Structure Theory and Interactions Workshop 2024 in Vienna.
In order to state them, we need the following definition.
For fixed $\M\models \PA$, let $\mathcal F$ be the set of first-order definable functions $f:
M\to M$ for which $x\leq f(x)\leq f(y)$ whenever $x\leq y$. Then for any $a\in M$
let the \define{gap} of $a$, $gap(a)$, be the smallest set $S$ with $a\in S$ and if $b\in S$,
$f\in\mc F$, and $b\leq x\leq f(b)$ or $x\leq b\leq f(x)$, then $x\in S$.
The gaps of $\M$
partition it into $\leq^\M$-intervals and we can thus obtain
an equivalence relation $a\sim_{gap} b \LR gap(a)=gap(b)$.
Unsurprisingly, in canonical extensions $\N_\mc{L}$ for $\N$ the prime model of its theory, the order type of the gaps is $1+ot(\mc{L})$~\cite[Corollary 3.3.6]{KS}.
\begin{proposition}
	For every sufficiently complex, realizable Scott complexity, and any completion $T$ of $\PA$, there is $\mc{M}\models T$ realizing the complexity so that $(\mc{M}){/}{\sim_{\text{gap}}}\cong 1+\eta$.
\end{proposition}

\begin{proof}
	Consider a colored linear ordering $\mc{L}$ with order type $\eta$. Then $\N_\mc{L}{/}{\sim_\text{gap}}\cong 1+\eta$.
    By \cref{prop:ComplexityTransfer}, all that remains is to show that we can find a colored linear ordering with order type $\eta$ of sufficiently large Scott complexity.

    Fix a linear ordering $\mc{L}$ with the desired Scott complexity and consider the colored linear ordering $\Psi(\mc{L})=(\eta_0+1_1+\eta_0)\cdot \mc{L}$.
    Note that the order type of $\Psi(\mc{L})$ is $\eta$ as it is dense without endpoints.
    We demonstrate that the Scott complexity of $\Psi(\mc{L})$ is the same as that of $\mc{L}$ if the complexity of $\mc{L}$ is sufficiently high.
    
    First, note that the image of $\Psi$ is Borel.
    To be in the image is to say that every element with color 1 is surrounded by densely many elements of color 0 (without endpoints).
    In other words, the order type is dense without endpoints, and the ordering satisfies the following $\Pinf{3}$-sentence:
    \[\forall x \big(P_1(x) \to (\exists y<x<z) P_0(y)\land P_0(z)\land (\forall w\neq x)  y<w<z\to P_0(w)\big)\land\]
    \[ (\forall v>x)  ((\forall v>v'>x) P_0(v')\to (\exists v''>v')  (\forall v''>v'''>v)  P_0(v''') ) \land \]
    \[ (\forall v<x ) ((\forall v<v'<x)  P_0(v')\to (\exists v''<v')  (\forall v''<v'''<v ) P_0(v''') )\big). \]
    
		Next, note that $\Psi$ has a computable, surjective inverse $\Phi$ that takes the order type of the $P_1$ elements.
        In particular, applying the pull-back theorem~\cite[Theorem XI.7]{MBook} to $\Phi$ allows us to conclude that for each $\varphi\in\Gamma_\alpha$ there is a $\varphi'\in\Gamma_\alpha$ such that
        \[\mc{L}\models\varphi \iff \Phi^{-1}(\Psi(\mc{L}))\models \varphi\iff \Psi(\mc{L})\models \varphi'.\]
        In the language of \cite[Proposition 2.21]{GR23}, this means that $\Psi$ "satisfies push-forward."
	Applying \cite[Proposition 2.21]{GR23} shows that complexities that are at least $\Pinf{3}$ are preserved by $\Psi$, giving the desired result.
    
\end{proof}

\begin{proposition}\label{prop:pinfomega2nohom}
    There is a model of Peano arithmetic of Scott rank $\omega+1$ that is neither homogeneous nor finitely generated.
\end{proposition}

\begin{proof}
    Let $\mc{L}=\eta_1+\eta_0+\eta_1$. This ordering has Scott rank 1 by \cref{thm:allComplexities} (in fact, it has Scott complexity $\Pinf{2})$.
    Therefore, by \cref{prop:ComplexityTransfer}, $\N_\mc{L}$ is of Scott rank $\omega+1$.

    We now show that $\N_\mc{L}$ is not finitely generated or homogeneous.
    Note that $\N_\mc{L}$ has a non-trivial automorphism group.
    To be more specific, every automorphism of $\mc{L}$ readily lifts to an automorphism of $\N_\mc{L}$.
    On the other hand, any finitely generated model of Peano arithmetic has a trivial automorphism group (see \cite{KS} Lemma 1.7.3).
    Therefore, $\N_\mc{L}$ is not finitely generated.
    $\N_\mc{L}$ is not homogeneous because there are elements $a$ and $b$ of type $p_1$ that are not in the same automorphism orbit. To see this simply pick $a$ from the first copy of $\eta_1$ and $b$ from the second copy.

\end{proof}

	\section{Missing complexities for models of PA}\label{sec:missing}
The methods described above fail to give us both qualitative and quantitative information about models of Peano arithmetic with Scott complexities $\Pinf\omega$, $\Pinf{\omega+1}$,$\dSinf{\omega+1}$, $\Sinf{\omega+1}$ and $\Sinf{\omega+2}$. To extract this part of the analysis, we turn to methods specific to arithmetic.
The central observation to analyze models of these complexities is that when the $\Sigma_n$-types of two tuples of elements of a model of $\PA$ coincide, their $\Sinf{n}$-types also coincide (see \cref{cor:typesEquiv}). This is a strengthening of a result by Montalbán and Rossegger~\cite[Lemma 5]{MR23} who noticed that this is the case when $n=\omega$. Combining this and generally well-understood structural properties of models of Peano arithmetic will allow us to calculate Scott complexities of certain well-known classes of models. 

In order to compactify the presentation of our results, we shall be working with \textit{pointed} models of $\PA$, to wit: models with a distinguished element. Please note that, by pairing, this is equivalent to working with finitely many fixed parameters.

\begin{proposition}\label{prop:baseCase}
    For $\mc{M}$, $\mc{N}$ pointed models of $\PA$, the following conditions are equivalent
    \begin{itemize}
        \item $\mc{M}\equiv_1 \mc{N}$;
        \item $\Th_{\Sigma_1}(\mc{M}) = \Th_{\Sigma_1}(\mc{N})$.
    \end{itemize}
\end{proposition}
\begin{proof}
	It follows from the MRDP theorem (see ~\cite[Result 7.8]{Kaye}) that any $\Sigma_1$ sentence about pointed models of $\PA$ is equivalent to a purely existential sentence. Furthermore, since $\mc M\equiv_1\mc N$ if and only if $\Th_{\Sinf{1}}(\mc M)=\Th_{\Sinf{1}}(\mc N)$ truth of a $\Sinf{1}$ sentence is determined by the truth of one of its (purely existential) disjuncts, $\Th_{\Sinf{1}}(\mc M)=\Th_{\Sinf{1}}(\mc N)$ if and only if $\Th_{\Sigma_1}(\mc M)=\Th_{\Sigma_1}(\mc N)$.
\end{proof}

The next theorem and particularly the corollary that follows are generalizations of a result known to hold when the models are recursively saturated~\cite[Proposition 1.8.1]{KS}.
\begin{theorem} \label{thm:SsyEqn} 
    For $2 \le n < \omega$ and any two pointed models $\mathcal{M}$ and $\mathcal{N}$ of $\PA$, the following conditions are equivalent:
    \begin{itemize}
        \item $\mathcal{M} \equiv_{n} \mathcal{N}$,
        \item $\Th_{\Sigma_{n}}(\mathcal{M}) = \Th_{\Sigma_{n}}(\mathcal{N})$ and $\SSy(\mathcal{M}) = \SSy(\mathcal{N})$.
    \end{itemize}
\end{theorem}
\begin{proof}
    Recall that for a set $A \subseteq \mathbb{N}$, $A \in \SSy(M)$ if there is $c \in M$ such that for all $n \in \mathbb{N}$, \[\underline{n} \in_{Ack} c \iff n \in A.\] 
	For a subset of $\mathcal{X} \subseteq \mathcal{P}(\mathbb{N})$ and any model $M \models \PA$ it follows that $\mathcal{X} = \SSy(M)$ if and only if \[A \in \mathcal{X} \iff M \models \exists x \left[ \left( \bigwedge_{n \in A} \underline{n} \in_{Ack} x \right) \wedge \left( \bigwedge_{n \in \mathbb{N} \setminus A} \underline{n} \not\in_{Ack} x \right)\right]\] where every such sentence is equivalent to a $\Sinf{2}$ sentence. 
	Now assume $n\geq 2$ and $\mc{M} \equiv_n \mc{N}$, then all $\Sinf{n}$ sentences true of $\mc{M}$ are true of $\mc{N}$ and vice-versa. In particular, this holds for $\Sinf{2}$ and $\Pinf{2}$ sentences and thus $\SSy(\mc{M}) = \SSy(\mc{N})$. Furthermore, by the MRDP theorem the $\Sigma_n$-theory of any model of $\PA$ is equivalent to its $\exists_n$-theory, with the latter being a subset of its $\Sinf{n}$ theory. Thus, $\Th_{\Sigma_n}(\mc{M})=\Th_{\Sigma_n}(\mc{N})$.
    
    On the other hand, suppose that $\SSy(\mathcal{M}) = \SSy(\mathcal{N})$. By induction on $n\geq 1$ we shall show that for all $\bar{a},\bar{b}$,
\[\Sigma_{n}\text{-}\tp_\M(\ba)=\Sigma_{n}\text{-}\tp_\N(\bb)\Longrightarrow (\mc{M},\bar{a}) \equiv_{n} (\mc{N},\bar{b})\]
    For $n=1$ this is true due to \cref{prop:baseCase}. Now suppose that if $$\Sigma_{n}\text{-}\tp_\M(\ba)=\Sigma_{n}\text{-}\tp_\N(\bb)$$ then $(\mc{M},\bar{a}) \equiv_{n} (\mc{N},\bar{b})$ and assume that $\Sigma_{n+1}\text{-}\tp_\M(\ba)=\Sigma_{n+1}\text{-}\tp_\N(\bb)$. We will show that $(\M,\ba)\leq_{n+1}(\N,\bb)$ using the game describing $\leq_{n+1}$ \cite[Section II.6]{MBook}. Say that the $\forall$-player plays $\bar d\in N^{<\omega}$. 
Then the $\Sigma_{n+1}$ type $p(\bar{x},\bar{y})$ of $\bb\bd$ over $\N$ is coded in $\N$ as it is clearly realized (see~\cite[Lemma 12.1]{Kaye}). Since $\SSy(\N)=\SSy(\M)$, that same type is also coded in $\M$. Furthermore, it is consistent in $(\M,\ba)$ by consistency in $(\N,\bb)$ and the assumption. These two facts together imply that $p(\ba,\bar{y})$ is realized in $\M$ (\cite[Lemma 12.2]{Kaye}) and hence there is an element $\bc$ such that $\Sigma_{n}\text{-}\tp_\M(\ba\bc)=\Sigma_{n}\text{-}\tp_\N(\bb\bd)$. By hypothesis $(\M,\ba\bc)\equiv_{n}(\N,\bb\bd)$ and thus $\bc$ is a winning play for the $\exists$-player. The proof that for such $\bb$ we have $(\N,\bb)\leq_{n+1}(\M,\ba)$ is symmetric.

Now to find such $\bb$ consider the type $$p_a = \Sigma_{n-1}\text{-}\tp_\M(\ba) \cup \Pi_{n-1}\text{-}\tp_\M(\ba).$$ Clearly it is a type over $\M$, i.e. $$\text{ for all finite } A \subseteq p_a \text{ we have }\mathcal{M} \models \exists \bar{x} \bigwedge_{\varphi \in A} \varphi(\bar{x}).$$ All such sentences are equivalently $\Sigma_n$ and since $\Th_{\Sigma_{n}}(\mathcal{M}) =\Th_{\Sigma_{n}}(\mathcal{N})$, it follows that they must be true in $\mc{N}$ and so $p_a$ is a type over $\mathcal{N}$.
    
        As $p_a$ is $\Sigma_{n}$ type realized in $\mathcal{M}$, it must also be coded. Since $\SSy(\mathcal{M})=\SSy(\mathcal{N})$, it is also a coded type in $\mathcal{N}$ and thus it is realized, i.e., for some $\overline{b}$ in $\mathcal{N}$ we have $\mathcal{N} \models p_a(\overline{b})$. 
        The argument that for all $\bar{b}\in\mc{N}$ there is $\bar{a}\in\mc{M}$ such that $(\mc{M},\bar{a})\equiv_{n-1}(\mc{N},\bar{b})$ is symmetric.
\end{proof}
The following reformulation of \cref{thm:SsyEqn} might be easiest to remember.
\begin{corollary}\label{cor:typesEquiv}
    For $1 \le n < \omega$, any two pointed models $\M$ and $\N$ of $\PA$ such that $\SSy(\M)=\SSy(\N)$ the following conditions are equivalent
    \begin{itemize}
        \item $\mc{M}\equiv_n \mc{N}$;
        \item $\Th_{\Sigma_n}(\mc{M}) = \Th_{\Sigma_n}(\mc{N}).$
    \end{itemize}
    
\end{corollary}

Recall that $\PA(\mathcal{L})$ denotes the extension of $\PA$ in a finite relational language $\mathcal{L}$ with finitely many constants that contains the full induction scheme for $\mathcal{L}$. The theorem below is an adaptation of \cref{thm:SsyEqn} to such extensions of $\PA$. Note that it is much weaker than in the $\PA$-case, because we do not have the MRDP theorem for expansions of the arithmetical signature.

    \begin{theorem} \label{thm:SsyEqn2} 
    For $2 \le n < \omega$ and any two models $\mathcal{M}$ and $\mathcal{N}$ of $\PA(\mathcal{L})$, the following implications hold:
    \begin{itemize}
        \item $\mathcal{M} \equiv_{\omega} \mathcal{N}\Longrightarrow \Th_{\mathcal{L}}(\mathcal{M}) = \Th_{\mathcal{L}}(\mathcal{N}) \textnormal{ and } \SSy(\mathcal{M}) = \SSy(\mathcal{N})$,
        \item $\Th_{\Sigma_{n}(\mathcal{L})}(\mathcal{M}) = \Th_{\Sigma_{n}(\mathcal{L})}(\mathcal{N}) \textnormal{ and }\SSy(\mathcal{M}) = \SSy(\mathcal{N})\Longrightarrow \mathcal{M} \equiv_{n} \mathcal{N}$.
    \end{itemize}
\end{theorem}
\begin{proof}
    The first implication follows, since each finitary formula can be seen as a particular infinitary formula, and the proof that the standard systems must be equal does not depend on the additional predicates from $\mathcal{L}$.
    
    The proof of the second implication is exactly as in Theorem \ref{thm:SsyEqn}: in $\PA(\mathcal{L})$ partial truth predicates for $\Sigma_n(\mathcal{L})$-formulae exist (see \cite[Section 1.11]{KS}) and the proof that "coded $\Sigma_n(\mathcal{L})/\Pi_n(\mathcal{L})$ types are realized" and "realized $\Sigma_n(\mathcal{L})/\Pi_n(\mathcal{L})$-types are coded" is as in the $\PA$ case (see \cite[Corollary 1.11.4]{KS}). 
\end{proof}

By taking unions on both sides of the equivalence in \cref{thm:SsyEqn} we obtain the following
\begin{corollary} \label{cor:SsyEq} 
    For any two models $\mathcal{M}$ and $\mathcal{N}$ of $\PA(\mathcal{L})$ the following conditions are equivalent:
    \begin{itemize}
        \item $\M \equiv_\omega \N$,
        \item $\Th(\M) = \Th(\N)$ and $\SSy(\mathcal{M}) = \SSy(\mathcal{N})$.
    \end{itemize}
    
\end{corollary}

\begin{corollary} \label{lem:omega+1Equiv} 
    For any two models $\mathcal{M}$ and $\mathcal{N}$ of $\PA(\mathcal{L})$, the following conditions are equivalent:
    \begin{itemize}
        \item $\mathcal{M} \equiv_{\omega+1} \mathcal{N}$,
        \item $\mathcal{M}$ and $\mathcal{N}$ realize the same types.
    \end{itemize}
\end{corollary}
\begin{proof}
    The top-down implication is clear as the satisfiability of a type can be expressed as a $\Sinf{\omega+1}$ sentence. We prove the bottom-up implication. Assume $\mathcal{M}$ and $\mathcal{N}$ realize the same types. Then clearly $\SSy(\mc{M}) = \SSy(\mc{N})$. Also for each $\bar{a}\in \mc{M}$ there is $\bar{b}\in \mc{N}$ such that $\Th(\mc{M},\bar{a}) = \Th(\mc{N},\bar{b})$. By the previous corollary we obtain that for every $\bar{a}\in M$ there is $\bar{b}\in N$ such that $(\mc{M},\bar{a})\equiv_{\omega}(\mc{N}, \bar{b})$.
			In particular, it follows that $\Sinf{\omega+1}-\Th(\mc{M})\subseteq \Sinf{\omega+1}-\Th(\mc{N})$. The other direction is proved symmetrically.
\end{proof}

\subsection{Characterization of prime and finitely generated models using Scott analysis}

In this section we assume that whenever $\mathcal{M}\models \PA(\mc{L})$, then a type ($\Sigma_n$-type, $\Pinf{\alpha}$-type, etc.) of an element $a\in M$ mean its type in $\mathcal{L}$ and we omit the reference to $\mc{L}$.

The following theorem directly answers Question 1.

\begin{theorem}\label{thm_srundefinMOPA}
    Let $\mathcal{M}\models \PA(\mathcal{L})$ and let $b\in \mathcal{M}$ be a non-definable element in $\mathcal{M}$. Then the Scott rank of $\mathcal{M}$ is strictly greater than $\omega$.
\end{theorem}
\begin{proof}
    Let $\mc{M},b$ be as in the statement of the theorem. We claim that the automorphism orbit of $b$ in $\mc{M}$ is not $\Sigma^{in}_{\omega}$-definable. Clearly, it is enough to show that for each $n$ the orbit is not $\Sigma_n^{in}$-definable. Fix $n$. By \cref{thm:SsyEqn2} it is enough to find $c\in M$ such that $\Sigma_{n}\text{-}\tp_{\mc{M}}(c) = \Sigma_n\text{-}\tp_{\mc{M}}(b)$, but $tp_{\mc{M}}(b)\neq tp_{\mc{M}}(c)$. By the Ehrenfeucht Lemma (see \cite[Theorem 1.7.2]{KS}) it is enough to find $c\neq b$ in $K(b)$ such that $\Sigma_{n}\text{-}\tp_{K(b)}(c) = \Sigma_n\text{-}\tp_{K(b)}(b)$. Work in $K(b)$. Since $b$ is not definable in $\mc{M}$ (hence by elementarity also in $K(b)$) the following set of formulae is a type over $K(b)$
    \[p(x,b):=\{x\neq b \wedge \phi(x):\phi(x)\in\Sigma_{n+1}, K(b)\models \phi(b)\}.\]
    Moreover, since $\Sigma_{n+1}\text{-} tp_{K(b)}(b)$ is coded, so is $p(x,b)$ as a finite extension of it. Hence, $p(x,b)$ is realized in $K(b)$ and consequently there is an element $c\in K(b)$ such that $\Sigma_{n}\text{-}\tp_{K(b)}(c) = \Sigma_n\text{-}\tp_{K(b)}(b)$ and $c\neq b$. 
\end{proof}

\begin{corollary}\label{cor:prime}
    Assume that $\mc{M}\models \PA(\mc{L})$. Then the following holds
    \begin{itemize}
        \item The Scott rank of $\mc{M}$ is $\omega$ iff $\mc{M}$ is a prime model of (a completion of) $\PA(\mc{L})$.
        \item The parametrized Scott rank of $\mc{M}$ is $\omega$ iff $\mc{M}$ is finitely generated.
    \end{itemize}
\end{corollary}

\subsection{Missing complexities}

The proposition below directly answers Question 3 from \cite{MR23}.

\begin{proposition}\label{cor_noPiomega}
	No model of $\mathrm{PA}(\mathcal{L})$ admits a $\Pi_\omega^{in}$ Scott sentence.
\end{proposition}
\begin{proof}
	By \cref{thm_srundefinMOPA} if $\mc{M}$ is a model of $\PA(\mathcal{L})$ with a $\Pinf{\omega}$ Scott sentence, then $\mc{M}$ must consists only of $\mc{L}$-definable elements. It is well-known that for any elementary end-extension $\N$, $\SSy(\M)=\SSy(\N)$. Let $\mc{N}$ be any proper elementary end-extension of $\mc{M}$. In particular, in $\mc{N}$ not every element is $\mc{L}$-definable, and hence $\mc{M}$ and $\mc{N}$ are not isomorphic. However, by \cref{cor:SsyEq}, $\N \equiv_\omega \M$ and hence if $\M$ had Scott complexity $\Pinf{\omega}$, then $\N$ would have to be isomorphic to $\M$, a contradiction.
\end{proof}

\begin{corollary}\label{cor_noSigmaomega+1}
    No model of $\PA(\mathcal{L})$ admits a $\Sinf{\omega+1}$ Scott sentence.
\end{corollary}
\begin{proof}
	If $\mc{M}$ has Scott complexity $\Sinf{\omega+1}$, then there is $a\in M$, such that $(\mc{M},a)$ has a $\Pinf{\omega}$ Scott sentence. By expanding $\mathcal{L}$ with a fresh constant, we can assume that $(\mc{M},a)\models \PA(\mathcal{L})$ and has a $\Pinf{\omega}$ Scott sentence, a contradiction with \cref{cor_noPiomega}.
    
\end{proof}
\begin{proposition}\label{prop_fingenchar}
  Models of $\PA(\mathcal{L})$ of Scott complexity $\dSinf{\omega+1}$ are exactly the finitely generated models which contain $\mathcal{L}$-undefinable elements.
\end{proposition}
\begin{proof}
    The only possibility for a structure $\mc{M}$ to have Scott complexity $\dSinf{\omega+1}$ is when its parameterized Scott rank is $\omega$ and its parameterless rank is $\omega+1$. Hence, by \cref{thm_srundefinMOPA} and \cref{cor:prime} if $\mc{M}$ is a model of $\PA(\mathcal{L})$, such $\mc{M}$ must be finitely generated and contain $\mathcal{L}$-undefinable elements.

		To see that every finitely generated model has a $\dSinf{\omega+1}$ Scott sentence, we will explicitly give one for a fixed finitely generated model $\mc{M}\models \PA(\mathcal{L})$. Let $p$ be the type of any element $a$ such that $\mc{M} = K(a)$ and for $b\in \M$ let $\psi_b$ be the formula so that $\psi_b(x,a)$ defines $b$ in $\mc M$. Let $\theta$ be the sentence
    \[\bigl(\exists x \bigwedge_{\phi(x)\in p}\phi(x)\bigr) \wedge \forall x \forall y \bigl(\bigwedge_{\phi(x)\in p}\phi(x)\rightarrow \bigvee_{\psi_b:b\in M} \psi_b(x,y)\wedge \forall z (\psi_b(x,z)\rightarrow z=y)\bigr).\]
    We claim that $\theta$ is a Scott sentence for $\mc{M}$. Take any $\mc{N}\models \theta$ and consider $b\in N$ which realizes $p$ (such an element exists by the first conjunct of $\theta$, we may assume that $\mc{N}\models \PA(\mathcal{L})$). By the second conjunct of $\theta$, every element of $N$ is first-order definable from $b$. Hence in particular, by the Ehrenfeucht Lemma, no element of $N\setminus\{b\}$ satisfies $p$. So there is an isomorphism between $\mc{M}$ and $\mc{N}$ that is uniquely determined by sending $a$ to $b$. 
    By \cref{cor_noSigmaomega+1} and \cref{thm_srundefinMOPA}, if $\mc{M}$ is not prime, then it can have neither a $\Sinf{\omega+1}$ nor a $\Pinf{\omega+1}$ Scott sentence. So $\theta$ is actually a Scott sentence for $\mc{M}$ of least possible complexity.
\end{proof}

\begin{corollary}\label{cor_noSigmaomega+2}
    No model of $\PA(\mathcal{L})$ has Scott complexity $\Sinf{\omega+2}$.
\end{corollary}
\begin{proof}
	If a model has Scott complexity $\Sinf{\omega+2}$, then it has parameterized Scott rank $\omega$. Hence, if a model of $\PA(\mathcal{L})$ has Scott complexity $\Sinf{\omega+2}$, then it must be finitely generated by  \cref{cor:prime}. However, such models have Scott complexity $d\text{-}\Sigma_{\omega+1}$ by \cref{prop_fingenchar}.
\end{proof}

Above we obtained a very neat characterization of models $\mc{M}\models\PA(\mathcal{L})$ whose Scott complexity is $\Pinf{\omega+1}$. Moreover, we have shown that the next realizable complexity is $\dSinf{\omega+1}$ and models of $\PA(\mathcal{L})$ from this class are exactly the finitely generated models. It is therefore natural to ask about other complexities. By \cref{cor_noSigmaomega+2}, the next two are $\Pinf{\omega+2}$ and $\dSinf{\omega+2}$. It is quite easy to deduce from the above results that all recursively saturated models of $\PA(\mathcal{L})$, being homogeneous, must have the Scott complexity exactly $\Pinf{\omega+2}$. However, \cref{prop:pinfomega2nohom} already delivers an example of a model which is not homogeneous but still has Scott complexity $\Pinf{\omega+2}.$ Below we show examples of homogeneous, not recursively saturated models of $\PA$ which also have Scott complexity $\Pinf{\omega+2}$, witnessing that this complexity class is much more diverse than the previous ones. Finally, we show that another natural class of nonstandard models of $\PA(\mathcal{L})$ all have Scott complexity $\dSinf{\omega+2}$.

Recall that a model $\mc{M}$ is \define{short recursively saturated} if and only if $\mc{M}$ realizes all the recursive bounded types, where $p(x,\bar{a})$ is a bounded type if and only if for some $c\in M$, $(x<c)\in p(x,\bar{a})$. If additionally $\mc{M}$ is nonstandard and not recursively saturated, then we say that it is \define{proper short recursively saturated}. By the result of \cite[Theorem 2.8]{Smor1981}, a countable model $\mc{M}\models \PA(\mathcal{L})$ is short recursively saturated if and only if $\mc{M}$ has an elementary end-extension to a recursively saturated model. 

We include a simple proof of the following proposition for completeness.

\begin{proposition}\label{prop:propshortrecsat}
    If $\mc{M}\models\PA(\mathcal{L})$ is nonstandard and short recursively saturated, then $\mc{M}$ is not finitely generated.
\end{proposition}
\begin{proof}
    Let $\mc{M}\models \PA(\mathcal{L})$ be nonstandard and short recursively saturated. Aiming at a contradiction, assume that $\mc{M}$ is finitely generated from an element $a\in M$. Consider the following recursive bounded type with parameter $a$
    \[p(x,a):= \{x<a \wedge (\ulcorner\phi(y)\urcorner\in x \equiv\phi(a)) : \phi(y)\in \mc{L}\}.\]
    We observe that $p(x,a)$ is finitely satisfied in $\mc{M}$ because finite sets of natural numbers can be coded by a natural number and any element realizing $p(x,a)$ encodes $\Th(\mc{M}, a)$. By the short recursive saturation of $\mc{M}$, $p(x,a)$ is realized, say by an element $c\in M$. Since $a$ generates $\mc{M}$, $c$ is definable from $a$, say via formula $\theta(x,y)$. Let us put
    \[T(x):= \exists y\bigl(\theta(a,y)\wedge x\in y\bigr).\]
    By unfolding the definitions, it is very easy to observe that for any $\phi(x)\in\mc{L}_{\PA}$ we have
    \[(\mc{M},a)\models \phi(a) \equiv T(\ulcorner{\phi(x)\urcorner}). \]
    In particular, $T(x)$ defines a truth predicate for $\Th(\mc{M},a)$ in $\Th(\mc{M},a)$, which contradicts Tarski's theorem on the undefinability of truth.
\end{proof}

\begin{proposition}
    Suppose that $\mc{M}\models\PA(\mathcal{L})$ is proper short recursively saturated and $K({\mc{M}})$ is cofinal in $\mc{M}$. Then the Scott complexity of $\mc{M}$ is $\Pinf{\omega+2}$.
\end{proposition}
\begin{proof}
    Fix $\mc{M}$ as above. By \cref{prop:propshortrecsat}, $\mc{M}$ is not finitely generated, so its Scott complexity is at least $\Pinf{\omega+2}$. However, $\mc{M}$ is homogeneous, as can be seen via the following simple argument: assume that $f:\mc{M}\rightarrow\mc{M}$ is a partial elementary map. Let $\mc{N}\succeq_e\mc{M}$ be an elementary end-extension of $\mc{M}$ which is recursively saturated. Then, by elementarity, $f:\mc{N}\rightarrow \mc{N}$ is a partial elementary map. Hence, by the homogeneity of $\mc{N}$, $f$ can be extended to an automorphism $f':\mc{N}\rightarrow\mc{N}$. However, since by elementarity $K({\mc{M}}) = K({\mc{N}})$ and $f'$ pointwise preserves $K({\mc{N}})$, $f'$ must preserve $\mc{M}$. So $f'$ restricts to an automorphism of $\mc{M}$.  Hence, the Scott complexity is at least $\Pinf{\omega+2}$, which matches the previously discussed lower bound.
    
\end{proof}

Below we show that, in general, proper short recursively saturated models have larger Scott complexity.

\begin{lemma}[{folklore; cf.~\cite[VII.2.11]{barwise1975}}] \label{lem:chains} 
    Suppose that $\{\mc{M}_{i}\}_{i\in\omega}$ is $\Pinf{\alpha}$ elementary chain. Let $\mc{M} = \bigcup_i \mc{M}_i$ and suppose that $\phi$ is a $\Pinf{\alpha+2}$ sentence which is satisfied by every $\mc{M}_i$. Then $\phi$ is true in $\mc{M}$.
\end{lemma}
\begin{proof}
	Assume towards a contradiction that $\M\models \neg \phi$. Then $\neg \phi\cong \bigvvee \exists \bar x\psi_j(\bar x)$ where $\psi_j(\bar x)\in \Pinf{\alpha+1}$. Let $\bar a \in \M$ and $j\in \omega$ be such that $\M\models \psi_j(\bar a)$ and as $\M=\bigcup_i \M_i$, let $i$ be so that $\bar a 
	\in\M_i$. Then as $\Pinf{\alpha+1}$ formulas are downwards absolute for $\Pinf{\alpha}$-elementary embeddings, $\M_{i}\models \psi_j(\bar a)$. But then $\M_i\models\neg \phi$, and by assumption $\M_i\models \phi$, a contradiction.
\end{proof}

\begin{proposition}
    Suppose that $\mc{M}\models\PA(\mathcal{L})$ is proper short recursively saturated such that $K(\mc{M})$ is not cofinal in $\mc{M}$. Then the Scott complexity of $\mc{M}$ is $d\text{-}\Sigma_{\omega+2}$.
\end{proposition}
\begin{proof}
    Assume that $\mc{M}$ is as in the proposition statement and let $p(x,a)$ be a recursive type which is not realized in $\mc{M}$. Then the $a$-definable elements are cofinal in $\mc{M}$. By our assumption that $K(\mc{M})$ is not cofinal in $\mc{M}$, $a$ must be above all the definable elements.  We claim that the isomorphism type of $\mc{M}$ is determined by (the conjunction of the infinitary sentences describing)
    \begin{itemize}
        \item[(a)] its theory;
        \item[(b)] its standard system;
        \item[(c)] the fact that $\mc{M}$ is short recursive saturated, that is
    \[\bigwedge_{e}\forall y\bigl(\bigwedge_{n\in\omega} \exists x\bigwedge_{i\in W_e, i<n}\phi_i(x,y)\rightarrow \exists x\bigwedge_{i\in W_e}\phi_i(x,y)\bigr),\]
    where the outermost conjunction ranges over indices of total Turing machines that compute bounded types.
    \item[(d)] the type of $a$ together with the information that $a$-definable elements are cofinal, that is
     \[\exists x \bigl(\bigwedge_{\phi(x)\in tp_{\mc{M}}(a)} \phi(x) \wedge\forall y\bigvee_{\psi(x,y)}\exists !z \psi(x,z)\wedge \forall z\bigl(\psi(x,z)\rightarrow y<z)\bigr)\bigr).\]
    \end{itemize}
		Call the conjunction of the above sentences $\theta$ (observe that $\theta$ can be taken to be $\dSinf{\omega+2}$). We claim that $\theta$ is a Scott sentence for $\mc{M}$.\footnote{This was first observed by Smoryński in \cite[Theorem 2.9]{Smor1981}.} To see this, assume that $\mc{M}\models \theta$ and $\mc{N}\models \theta$. Then $\mc{M}$ and $\mc{N}$ have the same theory, the same standard system, and are both proper short recursively saturated. Let $b\in N$ be the element which realizes $tp_{\mc{M}}(a)$ in $\mc{N}$ and such that $K(b)$ is cofinal in $\mc{N}$. By the result of Smoryński there are $\mc{M}\preceq_e \mc{M}'$ and $\mc{N}\preceq_e\mc{N}'$ such that $\mc{M}'$ and $\mc{N}'$ are both recursively saturated. Since $(\mc{M}',a)$ and $(\mc{N}',b)$ have the same theory, the same standard systems, and are recursively saturated there is an isomorphism $f:\mc{M}'\rightarrow \mc{N}'$ such that $f(a) = b$. Hence, we must have $f[M] = N$
        and  $f\upharpoonright_{M}$ is an isomorphism between $M$ and $N$. This witnesses that the Scott complexity of $\mc{M}$ is at most $\dSinf{\omega+2}$.

    To show that $\dSinf{\omega +2}$ is optimal, it is enough to exclude the existence of its $\Pinf{\omega+2}$ Scott sentence (because by \cref{cor_noSigmaomega+2} it cannot have  $\Sinf{\omega+2}$ Scott complexity.)  Fix any $\Pinf{\omega+2}$ sentence $\phi$ which holds in $\mc{M}$. Put $\mc{M}_0 = \mc{M}$ and let $\mc{M}_i$ be given. By induction we can assume that $\mc{M}_i$ is isomorphic to $\mc{M}$, so in particular $\mc{M}_i$ is short-recursively saturated and there is an element $a_i\in M_i$ such that $a_i$ realizes $tp_{\mc{M}}(a)$ in $\mc{M}_i$, $K(a_i)$ is cofinal in $\mc{M}_i$ and $a_i$ is above all definable elements. Let $\mc{M}_{i+1}$ be obtained from $\mc{M}_i$ in the following way: let $\mc{M}'$ be a proper elementary end-extension of $\mc{M}_i$ which is recursively saturated.  Let $a_{i+1}\in M'$ be any element realizing $tp_{\mc{M}}(a)$ in $\mc{M}'$ which is above $M_{i}$. Such an element exists because $a$ is above all the definable elements. Indeed, fix any $b\in M'\setminus M_i$ and consider the set of formulas
    \[p(x,b,a):= \{b<x \wedge \phi(x)\equiv \phi(a) : \phi(x)\in \mc{L}_{\PA}\}.\]
    It is sufficient to check that $p(x,a,b)$ is a type over $M'$. If it's not, the there are finitely many formulae $\phi_1(x),\ldots \phi_n(x)\in tp_{\mc{M}}(a)$ such that $\mc{M}'\models \forall x \bigl(\bigwedge_i\phi_i(x)\rightarrow x<b\bigr)$. So, by induction in $\mc{M}'$ there is the least $b$ such that $\mc{M}'\models \forall x \bigl(\bigwedge_i\phi_i(x)\rightarrow x<b\bigr)$. Such a $b$ is clearly definable, which implies that $a$ is below a definable element, a contradiction.

    We take $\mc{M}_{i+1}$ to be the initial segment of $\mc{M}'$ in which $K(a_{i+1})$ is cofinal. Observe that $\mc{M}_{i+1}\vDash \theta$, so $\mc{M}_{i+1}$ is isomorphic to $\mc{M}$. Now consider $\mc{N} = \bigcup_i\mc{M}_i$. Observe that for every $i$ and every $b\in M_{i}$ we have that $(\mc{M}_i,b)\equiv_{\omega}(\mc{M}_{i+1},b)$. This follows from \cref{cor:SsyEq}, since $\mc{M}_{i+1}$ is an elementary end-extension of $\mc{M}_i$, hence in particular both models have the same standard system. Since each $\mc{M}_i$ is isomorphic to $\mc{M}$, it follows that each $\mc{M}_i$ satisfies $\phi$ and hence by \cref{lem:chains}, $\mc{N}\models \phi$. However, $\mc{N}$ is not isomorphic to $\mc{M}$, because in $\mc{N}$ there is no last gap. This shows that $\phi$ cannot be a Scott sentence for $\mc{M}$.
\end{proof}

\printbibliography

\end{document}